\documentclass[12pt]{amsart}
\usepackage{latexsym,amssymb,amsmath, color, enumerate}
\newcommand{\RR}{{\mathbb R}}
\newcommand{\NN}{{\mathbb N}}
\newcommand{\ZZ}{{\mathbb Z}}

\newcommand{\CC}{{\mathbb C}}

\newcommand{\LL}{{\mathbb L}}

\newcommand\CIb{\mathcal C_b^\infty}

\newcommand{\cB}{{\mathcal B}}
\newcommand{\cC}{{\mathcal C}}
\newcommand{\cD}{{\mathcal D}}
\newcommand{\cE}{{\mathcal E}}
\newcommand{\cF}{{\mathcal F}}
\newcommand{\cG}{{\mathcal G}}

\newcommand{\cL}{{\mathcal L}}

\newcommand{\cO}{{\mathcal O}}
\newcommand{\cP}{{\mathcal P}}

\newcommand{\cS}{{\mathcal S}}

\newcommand{\cV}{{\mathcal V}}

\newcommand{\D}{{\partial }}

\newcommand{\bR}{\mathbb{R}}

\newcommand{\bP}{\mathbb{P}}
\newcommand{\bE}{\mathbb{E}}

\newcommand{\bL}{\mathbb{L}}

\newcommand{\bZ}{\mathbb{Z}}

\newcommand\ie{{\em i.e., }}

\newcommand\adj{\operatorname{ad}}

\newtheorem{theorem}{Theorem}[section]
\newtheorem{proposition}[theorem]{Proposition}
\newtheorem{corollary}[theorem]{Corollary}
\newtheorem{lemma}[theorem]{Lemma}
\newtheorem{definition}[theorem]{Definition}

\theoremstyle{remark}
\newtheorem{remark}[theorem]{Remark}

\numberwithin{equation}{section}

\newcommand{\beq}{\begin{equation}}
\newcommand{\eeq}{\end{equation}}

\newcommand{\field}[1]{\ensuremath{\mathbb{#1}}}
\newcommand{\R}{\ensuremath{\field{R}}}

\renewcommand{\sp}{{\ensuremath{{\scriptscriptstyle +}}}}

\newcommand{\<}{\langle}
\renewcommand{\>}{\rangle}

\newcommand{\s}{\sigma}

\definecolor{DarkRed}{rgb}{0.8,0.3,0.6}

\definecolor{Green}{rgb}{0.2,0.7,0.3}

\newcommand{\pa}{\partial}
\newcommand{\fA}{\mathfrak{A}}

\newcommand{\NAME}{Dyson-Taylor commutator method}

\author{Radu Constantinescu}
\email{radu.constantinescu@jpmorgan.com}
\address{Interest Rate Quantitative Research Group, JPMorganChase, New York, NY}

\author{Nick Costanzino}
\email{costanzi@math.psu.edu}

\author{Anna L. Mazzucato}
\email{mazzucat@math.psu.edu}

\author{Victor Nistor}
\email{nistor@math.psu.edu}
\address{Pennsylvania State University,
Math. Dept., University Park, PA 16802}

\thanks{A.M. was partially supported by NSF Grant DMS
0708902. V.N. was partially supported by NSF grant
DMS-0555831, DMS-0713743, and OCI 0749202}
\date{\today}

\begin{document}

\title[Parabolic Equations]{Approximate Solutions to Second Order
Parabolic Equations I: analytic estimates}

\begin{abstract}
We establish a new type of local asymptotic formula for the Green's
function $\cG_t(x,y)$ of a 
uniformly
parabolic linear operator $\pa_t - L$ with
non-constant coefficients using dilations and Taylor expansions at a
point $z=z(x,y)$, for a function $z$ with bounded derivatives such
that $z(x,x)=x \in \RR^N$. For $z(x,y) =x$, we recover the known, classical expansion obtained
via pseudo-differential calculus. Our method is based on
dilation at $z$, Dyson and Taylor series expansions, and the 
Baker-Campbell-Hausdorff commutator formula. Our procedure leads to
an elementary, algorithmic construction of approximate solutions
to parabolic equations which are accurate to arbitrary prescribed
order in the short-time limit. We establish 
mapping properties and
precise error estimates
in the exponentially weighted, $L^{p}$-type Sobolev spaces $W^{s,p}_a(\RR^N)$ 
that appear in practice.
\end{abstract}

\maketitle

\tableofcontents

\section{Introduction} \label{sec.intro}

We establish a new type of local estimate for the Green's function of
a uniformly parabolic linear operator with non-constant coefficients that {\em do not
depend on time}. More precisely, we consider second-order
differential operators $L$ of the form
\begin{equation}\label{eq.L}
	Lu(x) := \sum_{i,j=1}^N a_{ij}(x) \D_i \D_j u(x)+ \sum_{k=1}^N
	b_k(x)\D_k u(x)+ c(x)u(x),
\end{equation}
where $x = (x_1, ..., x_N) \in \bR^N$, $\D_k := \frac{\D}{\D x_k}$,
and the coefficients $a_{ij}$, $b_i$, and $c$ and all their
derivatives are assumed to be smooth and uniformly
bounded. (We then write $a_{ij}, b_j, c \in \cC_b^\infty\left(\bR^N
\right)$ and we denote the class of these operators by $\bL$.)
We also assume that $L$ is {\em uniformly strongly elliptic},
namely that there exists a constant $\gamma>0$ such that
\begin{equation}\label{eq.unif.s.ell}
  \sum_{ij} a_{ij}(x) \xi_i \xi_j \ge \gamma \|\xi\|^2 , \quad
  \|\xi\|^2 := \sum_{i=1}^ {N} \xi_i^2,
\end{equation}
for all $(\xi, x) \in\RR^N \times \RR^N$.  We define the matrix $A(x)
:= [a_{ij}(x)]$, which, without loss of generality, we can assume to
be symmetric. In view of the applications we are interested in, we take the
coefficients of $L$ to be real-valued. The set of operators $L \in
\bL$ satisfying \eqref{eq.unif.s.ell} will be denoted by
$\bL_\gamma$.

We study the {\em short time asymptotic} of the initial value problem (IVP) for the 
parabolic operator $\pa_t -L$:
\begin{equation}\label{eq.IVP}
\begin{cases}
  \D_t u(t, x) - Lu(t, x) = g(t, x) & \hspace{1.0cm} \mbox{in}\;
  (0,\infty)\times \bR^N\\ u(0, x) =f(x), &
\hspace{1.0cm} \mbox{on}\; \{0\}\times \bR^N\,,
\end{cases}
\end{equation}
for $u$, $f$, and $g$ in suitable function spaces.  
In view of Duhamel's principle, we may assume $g=0$.

When $b(x), c(x)\ne 0$ in \eqref{eq.L}, the corresponding parabolic equation
$\, \pa_t u- Lu =0$ is collectively referred to as a Fokker-Planck equation.
Fokker-Planck equations arise in many applications, for example in statistical
mechanics \cite{Carmichael,Gardiner}, and more generally in  probability.

We can also
replace $\R^N$ with a manifold of bounded geometry \cite{CCCMN,MN}, 
which thus allows us to treat also some {\em degenerate elliptic} operators $L$.
In particular, the approach in this paper can be extended to the case of
operators of the form $\partial_t - (a x^2\partial_x^2 + bx\partial_x + c)$ acting
on $\RR_t \times \RR_{x+}$ and to other operators that appear in practice. 
This extension is work in progress \cite{CCCMN}. See also below for a more 
detailed discussion of this point.

It is known that there exists $\cG^L \in \cC^\infty((0,\infty)
\times \bR^N \times \bR^N)$ such that
\begin{align}
	u(t,x) = \int_{\bR^N} \cG^L(t,x,y) f(y)dy, \quad t>0,
\end{align}
is a solution of the above equation, and it is unique if $f$ and $u$ satisfy 
certain growth conditions, specified later (see for
instance \cite{DiBenedetto}, page 237).  We will often write
$\cG^L(t,x,y)=\cG^L_t(x,y)$.  In case we have uniqueness, we shall
also use the notation $u(t) = e^{tL}f$.  The operator $e^{tL}$ is then
called the {\em solution operator} of the problem \eqref{eq.IVP}, and
its kernel $\cG^L_t$ the {\em Green's function}, or {\em fundamental
solution} of $L$, or {\em conditional probability density} in
applications to probability.

For $L$ with constant coefficients and for a few other cases, one can
explicitly compute the kernel $\cG^L$.  In general however, it is not
known how to provide explicit formulas for $\cG^L$, though there is a
large literature on developing methods to obtaining good asymptotic
formulas for the Green's function for $t$ small and $x$ close to
$y$. For example, interpreting the operator $L$ as a Laplace-Beltrami
operator on a manifold plus lower order terms, lead to formal
asymptotic expansions of the form
\begin{equation*}
  \cG_t(x,y) = \frac{e^{- \frac{d(x,y)^2}{4 t}}} {(4 \pi t)^{N/2}}
  \left( \cG^{(0)}(x,y) + \cG^{(1)}(x,y) t + \cG^{(2)}(x,y)t^n +
  \ldots \right),
\end{equation*}
as $t\to 0_+$, where $d(x,y)$ is the geodesic distance between $x$ and
$y$ and $\cG^{(j)}(x, y)$ are smooth functions in $x$ and $y$. Among
the vast literature we refer to \cite{Az, Hsu, Kampen, McKeanSinger,
Pleijel, V1, V1, Vas}, (see also \cite{Farkas, Greiner, Melrose2, TayPDEII} for a 
pseudo-differential operator perspective). However, one difficulty in the
practical implementation of this geometric approach is that, except
again in special cases, there is no closed form solution to the
geodesic equations used in defining $d(x,y)$, which thus needs to be
accurately approximated or computed numerically.

A related short-time asymptotic approach uses oscillatory type
integrals, which gives:
\begin{equation}
  \label{eq.statphase}
  \cG^L(t,x,y)\sim \sum_{j\geq 0} t^{(j-n)/2}
  p_j\left(x,t^{-1/2}(x-y)\right) e^{-\frac{(x-y)^T A(x)^{-1}
  \cdot(x-y)}{4t}},
\end{equation}
as  $t\to 0_+$, where $p_j(x,w)$ is a polynomial of
degree $j$ in $w$, and $A(x):=[a_{ij}(x)]$.  (We follow here Taylor
\cite[Chapter 7, Section 13]{TayPDEII}, where an asymptotic parametrix
for the heat equation on compact manifolds was constructed.)  Finally,
we mention the recent approaches in \cite{Ait} using multivariate
Hermite expansions, and in \cite{CFP} using an alternate construction of a parametrix approximation.

In our paper, we devise a new, elementary method to obtain asymptotic
expansions similar and even more general than
\eqref{eq.statphase}. Our method is based on dilating the coefficients
of $L$ around a point $z$ with ratio $t^{1/2}$, then expanding in a
Taylor series in $t$. We regard this expansion as a perturbation of
the operator $L_0$ obtained from $L$ by freezing coefficients at
$z$. We then use a Dyson-series perturbative expansion to approximate
the heat-kernel of $L$. The Dyson-series expansion turns out to be
explicitly computable using the Baker-Campbell-Hausdorff commutator
formula. We call the resulting method the {\em \NAME}.  We think that
our method is more accurate and more stable in practical implementations
\cite{CCLMN,CCCMN}.

The main goal is to provide an {\em explicit, algorithmic} method to
compute each term in the expansion, while at the same time obtain
sharp error bounds in both weighted and unweighted Sobolev spaces. We
do not work on compact manifolds, rather in $\RR^N$, so that the error
needs to be globally controlled (see below for a connection with
operators on non-compact manifolds of bounded geometry). In
particular, our approximation is valid uniformly in $x$ and $y$,
provided $t$ is small enough.

We think that our method, the \NAME, is more elementary than the ones 
found in the literature and since it relies on an
iterative time-ordered perturbative formula for the solution operator
$e^{tL}$, Equation \eqref{eq.perturbative}, a parabolic rescaling
argument, and a suitable Taylor's expansion of the coefficients of $L$
(equation \ref{eq.taylorexp}). Since the iterative formula is obtained
via repeated applications of Duhamel's principle, we could also treat
certain classes of semilinear equations following Kato's method, which
allows to take rougher data as well (see \cite{KT} in the context of
the Navier-Stokes equations).  We remark here that a similar parabolic
scaling combined with Taylor expansions has been used in obtaining a
short-time expansion for stochastic flows (see \cite{BenA,Cast}).

Our main result is the following theorem. We introduce the weight
$\<x\> = (1 + |x|^2)^{1/2}$.  Below, $W^{m,p}_a := W^{m,p}_a(\RR^N)$
is the exponentially weighted Sobolev space defined by
\begin{multline*}
    W_a^{m, p}(\RR^N) := \{u:\RR^N \to \CC,\ \partial^\alpha_x
    \big(e^{a\<x\>} u(\cdot)\big) \in L^p(\RR^N),\ |\alpha| \le m \},
\end{multline*}
for $1<p<\infty$, $m \in \ZZ_+$, and $a \in \RR$. (See also Equation
\eqref{def.w.S}).  When $a=0$, we recover the usual Sobolev spaces.
The need to consider exponentially weighted spaces arises in
applications to probability, in particular in stochastic volatility
models. For instance, after making the substitution $x = e^y$, the
payoff usually associated with the Black-Scholes equation (equation
\eqref{eq.BS} below) belongs to $W^{m,p}_a$ with $m=1$, $a<-1$, and
$p$ large.  We also denote
\begin{equation}\label{eq.def.G}
	G(z; x) = (4\pi)^{-N/2} \det(A(z))^{-1/2} e^{- x^T A(z)^{-1}
	x/4},
\end{equation}
where $z$ is a given point in $\RR^N$. It is interesting to mention
that the Black-Scholes equation fits into the framework of manifolds
with cylindrical ends, to which the results of Krainer \cite{Krainer}
apply. Manifolds with cylindrical ends are the simplest examples of
manifolds with bounded geometry.

To approximate the value of the Green function $\cG_t(x,y)$ at some
point $(x,y)$, we will use a Taylor-type expansion at the point $z$ of
a suitable parabolic rescaling of the coefficients of $L$, which,
however will be chosen depending of $x$ and $y$, $z=z(x,y)$.
Typically $z(x,y) =\lambda x + (1-\lambda)y$, for some fixed
$\lambda$, but we can allow more general choices. Namely, we shall say
that $z(x,y)$ is {\em admissible} if $z(x,x)=x$ and all derivatives
$\partial^\alpha z$ are bounded for $\alpha \neq 0$.

\begin{theorem} \label{theorem.main1}
Let $\mu\in \ZZ_+$, $L\in \LL_\gamma$, $z=z(x,y)$ be an admissible
function.  Let $\; \mathfrak{P}^\ell(z,x,y) = \sum a_{\alpha,
\beta}(z)(x-z)^\alpha(x-y)^{\beta}$, $|\alpha| \le \ell$, $\beta \le
3\ell$, $a_{\alpha, \beta} \in \CIb(\RR^N)$, be the functions provided
by the \NAME\ explained in the second half of this Introduction.
Define for each integer $0\leq \ell\leq \mu$, 
\begin{equation*} \cG^{[\mu,z]}_t(x,y) :=  t^{-N/2}
   \sum_{\ell=0}^{\mu} t^{\ell/2} \,
   \mathfrak{P}^\ell(\,z,\,z+\,\frac{x-z}{t^{1/2}}\,,
   \,z+\,\frac{y-z}{t^{1/2}}\,) G(z; \,\frac{x-y}{t^{1/2}}\,),
\end{equation*}
where $z=z(x,y)$. Define the error term $\cE_t^{[\mu,z]}$ in the
approximation of the Green's function by:
\begin{equation*}
             e^{tL} f(x) = \int_{\RR^N} \cG^{[\mu, z]}_t(x,y) f(y) dy
             + t^{(\mu+1)/2} \cE_t^{[\mu,z]} f(x).
\end{equation*}
Then, for any $f \in W_a^{m,p}(\RR^N)$, $a\in \RR$, $m\geq 0$, $1<p<\infty$,
we have
\begin{equation}
\label{eq.errorest}
   \|\cE_t^{[\mu,z]} f \|_{W_a^{m+k,p}} \le Ct^{-k/2}
   \|f\|_{W_a^{m,p}},
\end{equation} for any $t\in [0,T]$, $0<T<\infty$, 
$k\in \ZZ_+$, with $C$ independent of $t \in [0,T]$.
\end{theorem}

The function $\cG^{[\mu, z]}_t(x,y)$ will be called the {\em
$\mu$th-order approximation kernel} for the solution operator $e^{t
L}.$

See Subsection \ref{sec.GA} at the end of this introduction for a more
detailed description of how the approximation kernel $\cG^{[\mu,
z]}_t(x, y)$ is obtained.  In Section \ref{sec.commutator} we give an
explicit iterative construction of the functions $\mathfrak{P}^\ell$.
The main interest is, of course, in the derivation of the
approximation kernel. However, without good error estimates, this
kernel will not be of great use in practice.

For the particular choice $z(x,y) = x$, we have checked that the first
few polynomials $p_j(x,x-y):= \mathfrak{P}^j(x,x,y)$ coincide with the
ones in the expansion \eqref{eq.statphase} above (see \cite[Chapter 7,
Section 13]{TayPDEII}).  Our result is more general, however. We
discuss in \cite{CCCMN} different choices of the additional function
$z = z(x, y)$ in the framework of the usual Black-Scholes equation. It
turns out that the choice $z = x$ is not always the most
appropriate. In fact, for the Black-Scholes equation and $\mu=0$, the choice $z = \sqrt{xy}$ can lead to a better approximation,
whereas, surprisingly, the choice $z = (x + y)/2$ yields worse
numerical results than simply choosing $z=x$. This addional accuracy obtained for a suitable choice of
$z(x,y) \not = x$ and in view also of the mid-point quadrature
rule (which leads to a higher order of convergence) justifies
the extra generality of including arbitrary admissible $z$ in
our method.

A localization procedure as in \cite{MN} will allow us to pass from
operators on $\RR^N$ to operators on manifolds $M$ of bounded geometry
(again following \cite{CGT} and \cite{MN}). More precisely, our
results will extend to operators of the form $L = \sum_{ij}^N
a_{ij}\partial_i \partial_j + \sum_{ij}^N b_{i}\partial_i + c$,
defined on a subset $\Omega$ of $\RR^N$ such that the coefficients are
bounded in {\em normal coordinates} with respect to the metric $g =
\sum_{ij}^N a^{ij}dx_i dx_j$, assumed to be complete of bounded
geometry on $\Omega$. Here the matrix $[a^{ij}]$ is the inverse of the
matrix $A$, following the usual convention. Such metrics arise
naturally when resolving boundary singularities.  (see \cite{Melrose2}
for a systematic treatment of heat calculus on manifolds with
boundary).  We refer to \cite{ALN, CGT, Koch1, Trecent} for recent
papers dealing with partial differential equations on manifolds with
metrics of this form.  In particular, we can deal with certain
operators having polynomial coefficients such as those arising in
probability and its applications, for example in the Black-Scholes
option pricing equation \cite{BS}
\begin{equation} \label{eq.BS}
    Lu(x) = \sigma x^2\partial_x^2u(x)/2 + r(x\partial_xu(x) - u(x)),
\end{equation}
where in this context $t$ is the time to option expiry. Our results
also apply to differential operators arising in stochastic volatility
models (c.f. \cite{AL, BPV, FPS, Gatheral, Heston, Lesniewski,Lew}). On the
other hand, our results apply to operators of the form
$x^{2\beta}\partial_x^2$, $0< \beta < 1$ {\em only locally}. A good
framework for obtaining differential operators with unbounded
coefficients that satisfy our assumptions is that of Lie manifolds
\cite{ALN}. This point will be discussed in detail in \cite{CCCMN}.
Explicit calculations and concrete, practical applications of our
method will be given in \cite{CCLMN,CCCMN2}.

In addition, our methods generalize to operators with {\em
time-depdendent} coefficients, satisfying certain conditions. This
extension is addressed in a forthcoming paper \cite{CMN}.

We conclude this first part of the introduction with an outline of the
paper. In Section \ref{sec.prelim}, we define the weighted and regular
Sobolev spaces of initial data for the parabolic equation and
introduce the class of operators $L$ under study. We also briefly
discuss mapping properties of the semigroup generated by $L$ and use
them to justify the Dyson (or time-ordered) perturbation expansion of
$e^{tL}$. In Section \ref{sec.dilations}, we exploit local in space
and time dilations of the Green's function together with a certain
Taylor expansion of the operator $L$ to rewrite the perturbation
expansion as a formal power series in $s=\sqrt{t}$. In Section
\ref{sec.commutator}, we employ commutator estimates to derive
computable formulas for each term in the expansion.   This leads
to the \NAME\ to determine the functions $\mathfrak{P}^{\ell}$ used in
Theorem \ref{theorem.main1}.  

Finally, in Section~\ref{sec.EE}, we rigorously justify our expansion
and derive error bounds in time by means of pseudodifferential
calculus.\\

\noindent{\bf Acknowledgments} We thank Andrew Lesniewski and Michael
Taylor for sending us their papers and for useful discussions. We also thank Richard Melrose for carefully reading our paper and for
pointing out a possible missunderstanding in an earlier version of
this paper.

\subsection{The approximate Green function\label{sec.GA}}

We close this Introduction by describing in more detail the \NAME used to define approximation kernel $\cG^{[n, z]}_t$.  Given an
operator $T$ with smooth kernel, we denote its kernel by $T(x, y)$, as
customary.

We consider a {\em uniformly strongly elliptic} differential operator
$L$ of the form $L := \sum_{i,j=1}^N a_{ij} \D_i \D_j + \sum_{k=1}^N
b_k\D_k + c$, where $a_{ij}, b_k, c \in C^\infty_b(\RR^N)$ are {\em
real valued}.  Given a fixed point $z \in \RR^N$ and $s>0$, we define
$L^{s,z} := \sum_{i,j=1}^N a^{s,z}_{ij}(x) \D_i \D_j + s \sum_{i=1}^N
b^{s,z}_i(x) \D_i + s^2 c^{s,z}(x)$, where for a generic function $f$
we set $f^{s, z}(x) = f(z + s(x-z))$.  Hence, $z$ acts as a fixed
dilation center.  We then Taylor expand this operator in $s$ around
$0$ to order $n$:
\begin{equation} \label{eq.taylorexp)}
  L^{s,z} = \sum_{m=0}^n s^m L^{z}_m + V_{n+1}^{s,z} =
  \sum_{m=0}^{n+1} s^m L^{z}_m,
\end{equation}
where $L^{z}_j$, $0 \le j \le n$, are differential operators with {\em
polynomial coefficients} that do not depend on $s$, whereas
$L^{z}_{n+1}$ has smooth coefficients that, however, do
depend on $s$ (although this dependence is not shown in the notation). Hence,
$V_{n+1}^{s,z} =s^{n+1} L^{z}_{n+1}$ is the remainder of the Taylor
expansion.
The order $n$ will be chosen later. In particular, we observe that
\begin{equation} \label{eq.L0def}
     L^z_0 =  \sum_{i,j} a_{ij}(z) \pa_i\pa_j.
\end{equation}

For any fixed, positive integers $k \le n+1$ and $\ell$, we shall
denote by $\fA_{k,\ell}$ the set of multi-indexes $\alpha = (\alpha_1,
\alpha_2, \ldots, \alpha_k) \in \NN^k$, such that $|\alpha| := \sum
\alpha_j = \ell$. (Hence, $k\leq \ell$.) Then, for each multi-index $\alpha = (\alpha_1,
\ldots, \alpha_k) \in \fA_{k,\ell}$, $k \le n$, we introduce
\begin{equation} \label{eq.Lambda_alpha1}
    \Lambda_{\alpha,z} := \int_{\Sigma_k} e^{\tau_0 L_0^{z}}
    L^{z}_{\alpha_1} e^{\tau_1 L_0^{z}} L^{z}_{\alpha_2} \cdots
    L^{z}_{\alpha_k} e^{\tau_k L_0^{z}} d\tau.
\end{equation}
Observe that $k,\ell$ are unique given $\alpha$.

The main point is that the operators $\Lambda_\alpha$ can be computed
explicitly as follows.  Let us denote by $\cD(a,b)$ the vector space
of all differentiations of polynomial degree at most $a$ and order at
most $b$.  (By {\em polynomial degree} of a differentiation we mean the highest power of
the polynomials appearing as coefficients.)  Then, for any $L_0
\in \cD(0,2)$ that is uniformly strongly elliptic and for any $L_m \in
\cD(m,2)$, we have a differential operator $P_m(L_0, L_m;
\theta,x,\D)$ given by the formula $e^{\theta L_0} L_m = P_m(L_0, L_m;
\theta,x,\D) e^{\theta L_0},$ where $\theta > 0$ (see Lemma
\ref{lemma.finite-BCH}). Let $\Sigma_k$ be the unit $k$-dimensional
simplex.  Next, for any given multi-index $\alpha \in \fA_{k,\ell}$
with $k \le n$, we define ${\cP}_{\alpha}(x, z, \D) := \int_{\Sigma_k}
\prod_{i=1}^k P_{\alpha_i} (L_0^{z}, L_{\alpha_i}^{z}; 1-\sigma_i, x,
\D) d\sigma$. Then
\begin{equation} \label{eq.lambdaalphadef}
  \Lambda_{\alpha, z} = \cP_{\alpha}(x,z,\D_{x})e^{L_0^{z}},
\end{equation}
where the product is the composition of operators and $\cP_{\alpha}$
is a differential operator of order $2k+ \ell$ in
$x$ and polynomial degree $\leq \ell$ in $(x-z)$
(see Lemma \ref{lemma.explicit}).

Since $z$ is arbitrary, but fixed at this stage, if $L_0^z$ is the operator
in \eqref{eq.L0def} then $e^{L_0^z}(x,y)$ can be explicitly calculated and it agrees
with the function $G(z,x-y)$ introduced in equation \eqref{eq.def.G}. Therefore, it can
be easily seen from \eqref{eq.lambdaalphadef}, that
\[
     \Lambda_{\alpha, z}(x,y) = \mathfrak{P}^\ell(z, x,y) G(z; x-y),
\]
for some $\; \mathfrak{P}^\ell(z,x,y) = \sum a_{\alpha, \beta}(z)(x-z)^\alpha(x-y)^{\beta}$,
$|\alpha| \le \ell$, $\beta \le 3\ell$, $a_{\alpha, \beta} \in \CIb(\RR^N)$. In particular,
all $\Lambda_{\alpha,z}$ are operators with smooth kernels, thus denoted
$\Lambda_{\alpha,z}(x, y)$.

We will show that $e^{tL}$ as well is an operator
with smooth kernels, henceforth denoted $\cG^{L}_t(x, y)$.

Let us fix for the time being a smooth function $z : \RR^N \times
\RR^N \to \RR^N$ whose properties will be made precise below. (Two
typical examples are $z(x, y) = (x+y)/2$ and $z(x, y) = x$, which
suffice in many applications.)

Our approximation will be obtained by combining Lemma
\ref{lemma.Green.funct.cor} with the perturbative estimate
of Equation \eqref{eq.L.pert} at some point $z=z(x,y)$
using the dilation with center $z(x, y)$ and denoting $s^2 = t$.
Then, for any $\mu\geq n$, we define
\begin{multline} \label{eq.main.def}
	\cG_t^{[\mu,z]}(x,y) := s^{-N} e^{L_0^{z}}(z + s^{-1}(x - z), z
  	 + s^{-1}(y-z)) \\ + \sum_{\ell = 1}^\mu  \sum_ {k=1}^{ \ell }
  	 \sum_{\alpha \in \fA_{k, \ell}} s^\ell \Lambda_{\alpha, z}(z
  	 + s^{-1}(x - z), z + s^{-1}(y-z)) .
\end{multline}
The operator $L$ is not shown explicitly in the notation $\cG_t^{[\mu, z]}(x,y)$,
although $\cG_t^{[\mu,z]}(x,y)$ does depend on $L$. This is not likely
to cause any confusion, since $L$ is usually fixed in our discussions.

The justification of the above definition for the approximation is
that a Dyson series expansion of order $n+1$ gives us
\begin{multline} \label{eq.full.sum}
	\cG_t(x,y) := s^{-N} e^{L_0^{z}}(z + s^{-1}(x - z), z +
  	 s^{-1}(y-z)) \\ + \sum_{\ell = 1}^{(n+1)^2} \ \sum_ {k=1}^{
  	 \max\{\ell,n+1\} } \ \sum_{\alpha \in \fA_{k, \ell}} s^\ell
  	 \Lambda_{\alpha, z}(z + s^{-1}(x - z), z + s^{-1}(y-z)),
\end{multline}
where for $k=n+1$ and $\alpha = (\alpha_1, \ldots, \alpha_k) \in
\fA_{k,\ell}$, we introduce
\begin{equation} \label{eq.Lambda_alpha2}
    \Lambda_{\alpha,z} := \int_{\Sigma_k} e^{\tau_0 L_0^{z}}
    L^{z}_{\alpha_1} e^{\tau_1 L_0^{z}} L^{z}_{\alpha_2} \cdots
    L^{z}_{\alpha_k} e^{\tau_k L^{s,z}}  d\tau.
\end{equation}
The difference between equations \eqref{eq.main.def} and
\eqref{eq.full.sum} is that the sum in the first equation contains
exactly the terms with $s^{\ell}$, $\ell \le n$, from the second
equation.  The difference between equations \eqref{eq.Lambda_alpha1}
and \eqref{eq.Lambda_alpha2} is in the last exponential. Note that
$\Lambda_{\alpha, z}$ does not depend on $s$ if $\ell = |\alpha| \le
n$, but it may depend on $s$ otherwise.  In any case, all the terms
$\Lambda_{\alpha, z}$ that depend on $s$ will be included in the error
term.
All terms $\Lambda_{\alpha, z}$ with $\mu<|\alpha|\leq n$, which  do not depend on $s$, will
also be included in the error. We remark that the error term is never computed explicitly,
as only the $\mu$th order approximation kernel is needed. Therefore, while $\mu$ will usually be small in
applications, we can take $n$ as large as needed to justify the error bounds of Theorem \ref{theorem.main1}.
In Section \ref{sec.EE}, we will show that $n> \mu+N-1$ suffices.

\section{Preliminaries} \label{sec.prelim}

We begin by discussing in more details the class of second-order
operators $L$ of the form \eqref{eq.L} that are the focus of our
work. Below we set \beq \label{eq.def.bdderiv} C^\infty_b(\RR^N) := \{
f: \RR^N \to \CC, \partial^\alpha f \text{ bounded for all }
\alpha\,\}.  \eeq

\begin{definition}\label{def.Lgamma}
We shall denote by $\bL$ the set of differential operators $L$ of the
form
\begin{equation}\label{eq.L2}
	L := \sum_{i,j=1}^N a_{ij} \D_i \D_j + \sum_{k=1}^N b_k\D_k +
	c,
\end{equation}
where $a_{ij}, b_k, c \in C^\infty_b(\RR^N)$ are {\em real valued}.
We shall denote by $\bL_\gamma$ the subset of operators $L \in \bL$
satisfying the uniform strong ellipticity estimate
\eqref{eq.unif.s.ell} with the ellipticity constant $\gamma$.  We let
$A=[a_{ij}]$ and assume additionally that $A$ is symmetric,
which can be achieved simply by replacing $A$ with its
symmetric part, since this does not change our differential operator.
\end{definition}

The above definition can be extended to operators on manifolds of
bounded geometry $M$ (see \cite{CGT, MN, Shubin}). For example, when
$M = \RR^N$ with the Euclidean metric, the class $\bL$ considered in
\cite{MN} coincides with the class $\bL$ considered in this paper.

In what follows, we denote the inner product on $L^2(\RR^N)$ by $(u, v) =
\int_{\RR^N} u(x) \overline{v(x)} dx$. Let us denote $\langle \xi \rangle :=
(1+|\xi|^2)^{1/2}$ and let $\Hat{u}$ be the Fourier Transform of
$u$. We also recall the definition of and some basic facts about
$L^p$-based Sobolev spaces $W^{r,p}(\RR^N)$ . For $1 < p < \infty$,
$r\in \RR$:
\begin{multline}\label{eq.def.wsR}
    	W^{r,p}(\RR^N):=\{u:\RR^N\to \CC\, , \ \langle \xi \rangle^r
	\Hat{u} \in L^p(\RR^N)\} \\ = W^{r, p}(\RR^N) := \{ u:\RR^N
	\to \CC\, ,\ (1-\Delta)^{r/2} u \in L^p(\RR^N) \},
\end{multline}
If $r\in \ZZ_+$,
\[
 W^{r,p}(\RR^N)= \{u:\RR^N \to \CC,\ \partial^\alpha u \in
L^p(\RR^N),\ |\alpha| \le r \}.
\]
Since the dimension $N$ is fixed throughout the paper, we will usually
write $W^{r,p}$ for $W^{r,p}(\RR^N)$.  When $1<p<\infty$, the dual of
$W^{r,p}$ is the Sobolev space $W^{-r,p'}$ with $1/p+1/p'=1$.

We are interested in considering the initial value problem
\eqref{eq.IVP} in the largest-possible space of initial data $f$ that
includes the typical initial conditions that arise in applications and
where uniqueness holds. We therefore introduce exponentially weighted
Sobolev spaces.  Given a fixed point $y\in \RR^N$, we set $\<x\>_{w} :=
\<x-w\>= (1 + |x - w|^2)^{1/2}$ and define $W_{a,w}^{m, p}(\RR^N)$ for
$m \in \ZZ_+$, $a\in \RR$, $1<p<\infty$, by
\begin{multline}\label{def.w.S}
    W_{a, w}^{r, p}(\RR^N)
    := e^{-a\<x\>_{w}} W^{r, p}(\RR^N) \\ =
    \{u:\RR^N \to \CC,\ \partial^\alpha_x \big(e^{a\<x\>_{w}}
    u(\cdot)\big), \in L^p(\RR^N),\ |\alpha| \le r \}, \quad \text{if
    } r \in \ZZ_+,
\end{multline}
with norm
\begin{equation*}
        \|u\|_{W_{a, w}^{m, p}}^p := \| e^{a\<x\>_{w}} u\|_{W^{m,p}}^p
        = \sum_{|\alpha| \le m} \|\partial^\alpha_\xi
        \big(e^{a\<x\>_{w}} u(x)\big)\|_{L^p}^p.
\end{equation*}
When it is clear from the context, we may drop the subscript $w$
from the above notation.  We observe that $W_{0}^{m,p} =
W_{0,w}^{m,p} = W^{m,p}$. The spaces $W^{r,p}_{a,w}$ and
$W^{-r,p'}_{-a,w}$ are naturally duals to each other if $1/p+1/p' =1$.

A crucial observation is that, for any $L \in \bL_\gamma$ and any $a
\in \RR$, the operators $L_a := e^{a\<x \>_{w}} L e^{-a\<x\>_{w}}$ are
also in $\bL_\gamma$. They moreover define a
bounded family in $\mathbb{L}_{\gamma}$ if $a$ is in a bounded set,
while $w$ is {\em arbitrary}.  Since proving a result for $L$ acting
between weighted Sobolev spaces $W^{s,p}_{a, w}$ is the same as
proving the corresponding result for $L_1$ acting between the Sobolev
spaces $W^{s,p} = W^{s,p}_{0, w}$, we may assume that $a=0$ and $w$ is
arbitrary.  In particular, $L : W^{s+2,p}_{a ,y} \to W^{s,p}_{a, w}$
is well defined and continuous for any $a$ and $w$, since this is true
for $a=0$.

In fact, it will be crucial for us to establish mapping properties
that are {\em independent} of $w$. This will be the case in all
estimates below, unless stated otherwise. One of the most important
example is provided by Corollary \ref{cor.wel.def}.  Moreover, the
spaces $W^{m,p}_{a, w}$ do not depend on the choice of the point $y$
(although their norm obviously does).  Because of this observation, we
shall often omit the point $w$ from the notation, when this does not
affect the clarity of the presentation.

We begin by recalling some properties of $L$ and the associated
solution operator $e^{tL}$ to the initial value problem
\eqref{eq.IVP}.

\subsection{Mapping properties}
Given a Banach $X$ and an interval $I$ of the real line, we shall
denote by $\cC(I, X)$ the space of continuous functions $u: I \to
X$. By $\cC^k(I, X)$ we shall denote the space of functions $u \in
\cC(I, X)$ such that $u^{(j)} \in \cC(I, X)$ for all $0 \le j \le
k$. We assume that $X \subset L^1_{\operatorname{loc}}(\RR^N)$, and
that $L$ is a closed unbounded operator on $X$ with domain $\cD(L)
\subset X$.

Let $g \in \cC([0,\infty), X)$.  By a {\em classical solution in $X$}
of \eqref{eq.IVP} we mean a function
\beq\label{eq.def.cs} u \in \cC([0,\infty), X) \cap
    \cC^1((0,\infty),X) \cap \cC((0,\infty),\cD(L)), \eeq
such that $\pa_t u(t) = L u(t) + g(t)$ in $X$ for all $t>0$ and $u(0)
= f$ in $X$.  (The domain of $L$ is given the graph norm $|||u||| :=
\|u\| + \|Lu\|$, which makes $\cD(L)$ a complete normed space, since
we have assumed that $L$ is closed and $X$ is complete.) In
particular, $u(0)=f$ must belong to the closure of $\cD(L)$ in $X$. In
the case of interest here, if $X=W^{s,p}_a$, then $\cD(L) =
W^{s+2,p}_a$, which is dense in $X$.

In view of Duhamel's formula (which will be justified below), we can
assume $g=0$ in Equation \eqref{eq.IVP}. We shall take our Banach
space where the solution is defined to be $X = L^p_a$ for some
arbitrary, but fixed, $p \in (1, \infty)$ and $a \ge 0$. Then Equation
\eqref{eq.IVP} becomes
\begin{equation}\label{eq.IVP0}
\begin{cases}
    \D_t u(t) - Lu(t) = 0& \hspace{1.0cm} \mbox{in}\; L^p_{a}(\RR^N),
    \\ u(0) =f & \hspace{1.0cm} f \in L^p_a(\R^N).
\end{cases}
\end{equation}
Let us notice that if $f \in \cC^\infty(\RR^N)$ also, then we recover
Equation \eqref{eq.IVP}. The growth condition $u(t) \in L^p_a$ is
needed, however, in order to insure uniqueness.

A family of (bounded) linear operators $U(t)$ on $X$, $t\geq 0$, will
be called a $\cC^0$ or {\em strongly continuous} semigroups of
operators if $U(t)$ forms a semigroup in $t$ and $U(t)u \to u$ in $X$
as $t\to 0+$. This last property shows that the function $[0, \infty)
\ni t \to U(t)f \in X$ is continuous for any $f\in X$.

 We shall need the
following standard result.  Recall the subset $\bL_\gamma \subset \bL$
introduced in Definition \ref{def.Lgamma}.

\begin{lemma}\label{lemma.st1}
(i)\ Let $L\in \bL_\gamma$, then there exists a constant $C > 0$ such
that
\[
	\gamma(\nabla u, \nabla u) - C(u, u) \le -(Lu,u) \le C(\nabla
	u,\nabla u) + C(u, u).
\]
\noindent (ii)\ The norm $|||v|||_{2m} := \|u\|_{L^p} + \|L^m u
\|_{L^p}$ is equivalent to the norm $\| \; \cdot \: \|_{W^{2m, p}}$ on
$W^{2m,p}(\RR^N)$, for any $m\in \ZZ_+$ and $1<p<\infty$.
\end{lemma}

\begin{proof} (Sketch.) (i) follows from a direct calculation.
See \cite{Shubin}, \cite{Tay}
or \cite{MN} for (ii).
\end{proof}

It follows from this lemma that $L : W^{2,p} \to L^p$ is a closed,
densely defined unbounded operator on $L^p$. This technical fact is
important because it is often needed for the general results that we
will use below.

For the sake of clarity and completeness, we include here a quick
review and some proofs of the main properties of the semigroup
generated by $L$. Our proofs also serve the purpose of justifying the
perturbative expansion described in Section \ref{ssec.perturbative},
which is discussed extensively in the literature, but usually not in
the setting that we need. Further details can be found in
\cite{Lunardi,KatoBook, Pazy}. Below $p \in (1,\infty)$ and $\gamma>0$
will be arbitrary but fixed, and the constants appearing in the
estimates depend on $p$ and $\gamma$, but not on $L \in \bL_\gamma$.

\begin{proposition}\label{prop.st2} Let $a\in \RR$, $1<p<\infty$, and
$L \in \bL_\gamma$.
\begin{enumerate}[(i)]
\item For each $f\in W^{2,p}_a$, the problem \eqref{eq.IVP} has a
unique classical solution \label{it.prop.st2.i}
\begin{equation*}
  u \in \cC([0,\infty), L^p_a) \cap \cC^1((0,\infty), L^p_a) \cap
   \cC((0,\infty),W^{2,p}_a).
\end{equation*}
\item
Let $e^{tL}f := u(t)$, then we have $e^{tL} W^{r, p}_a \subset W^{r,
p}_a$ and, moreover, $\|e^{tL}f\|_{W^{r,p}_a} \le C e^{\omega t}
\|f\|_{W^{r,p}_a}$, for a constant $C$ independent of $r$, $a$, and $L
\in \bL_\gamma$ in {\em bounded} sets..
\label{it.prop.st2.ii}
\end{enumerate}
\end{proposition}

\begin{proof} We can assume $a=0$, as explained above.
Lemma \ref{lemma.st1} (i) gives that $L$ satisfies the assumptions of
the Hille-Yosida theorem \cite{Evans, Lunardi, Pazy}, and hence
$e^{tL}$ is defined, is a $C^0$ semigroup, and $u(t) := e^{tL}f$ is
indeed a classical solution. This proves (i).

It also follows from standard properties of $C^0$-semigroups in Banach
spaces that $\|e^{tL}f\|_{L^p} \le C e^{\omega t}\|f\|_{L^p}$ for some
constants $C>0$ and $\omega \in \RR$ independent of $L \in
\bL_\gamma$. To prove (ii), we then notice that
\begin{multline}\label{eq.estimate1}
    \| e^{tL}f \|_{W^{2m,p}} \le C ||| e^{tL}f |||_{2m,p} = C \big( \|
    e^{tL}f \|_{L^p} + \| L^m e^{tL}f \|_{L^p} \big) \\ = C \big( \|
    e^{tL}f \|_{L^p} + \| e^{tL} L^m f \|_{L^p} \big) \le C e^{\omega
    t} \big( \| f \|_{L^p} + \| L^m f \|_{L^p} \big) \\ = C e^{\omega
    t} ||| f |||_{2m} \le C e^{\omega t} \| f \|_{W^{2m,p}},
\end{multline}
with constants depending on $m$, $p$, and $L$, but not on $t$.  Though
$L$ may not be self-adjoint, the adjoint $L^\ast$ is an operator of
the same type, in the sense that $L^\ast \in \bL_\gamma$. Hence the
estimate above holds for $L^\ast$, with possibly different
constants. We can then extend Equation \eqref{eq.estimate1} to
$W^{-2m,p'}$, $m\in \ZZ_+$, by duality and to any $W^{r,p}$ by
interpolation (see for example \cite{BL, T} for results on
interpolation).  This completes the proof.
\end{proof}

{From} now on we shall denote by $e^{tL}$ the $\cC^0$-semigroup
generated by $L$ on $L^p_a= e^{a\<x\>_{z}} L^p(\RR^N)$, with $p$ and
$a$ determined by the context (usually arbitrary, but fixed).

We recall that for $f \in \cD(L)$, the map $t \to e^{tL}f$ is in
$\cC^1([0,\infty), X)$ and $\pa_t e^{tL}f = e^{tL}Lf = Le^{tL}f$.  For
any two normed spaces $X$ and $Y$, we denote by $\cB(X, Y)$ the normed
space of continuous, linear operators $T: X \to Y$ with norm $\|T\|_{X
\to Y}$.  When $X=Y$, we shall also write $\|T\|_{X} := \|T\|_{X \to
X}$ and $\cB(X) := \cB(X, X)$. The identity operator of any space will
be denoted by $1$.

\begin{lemma}\label{lemma.cont}\ Let $L \in \bL_\gamma$. We have
$\|e^{tL} - 1\|_{W^{s+2, p}_a \to W^{s, p}_a} \le Ct$, for any $t \in
(0, 1]$.  In particular, $[0, \infty) \ni t \to e^{tL} \in \cB(W^{s+2,
p}_a,W^{s, p}_a)$ is continuous.
\end{lemma}

\begin{proof}
We have $e^{tL}f - f = \int_{0}^t e^{sL}Lf ds$ for any $f \in
W^{2,p}_a$, by standard properties of $\cC^0$-semigroups. Lemma
\ref{prop.st2} \eqref{it.prop.st2.ii} then gives
\begin{equation*}
    \|e^{tL}f - f\|_{W^{s,p}_a} \le \int_{0}^t \|e^{sL}\|_{W^{s,p}_a}
    \|L f\|_{W^{s,p}_a} ds \le C te^{\omega t} \|f\|_{W^{s+2,p}_a},
\end{equation*}
which proves the first part of the result.

Let now $t_1 \ge t_2$. Then
\begin{equation*}
    \|e^{t_1L} - e^{t_2L}\|_{W^{s+2, p}_a \to W^{s, p}_a} \le
    \|e^{(t_1- t_2) L} - 1\|_{W^{s+2, p}_a \to W^{s, p}_a}
    \|e^{t_2L}\|_{W^{s, p}_a}.
\end{equation*}
This completes the second part of the proof.
\end{proof}

\begin{remark}
Let $\delta \in (0, 2]$. Then an interpolation argument gives
$\|e^{tL} - 1\|_{W^{s+\delta, p}_a \to W^{s, p}_a} \le Ct^{\delta/2}$,
for any $t \in (0, 1]$. Hence the function $[0, \infty) \ni t \to
e^{tL} \in \cB(W^{s+\delta, p}_a,W^{s, p}_a)$ is also continuous.
\end{remark}

We discuss smoothing properties of $e^{tL}$, it is convenient to first
assume $L^* = L$, that is that $L$ is self-adjoint. This will require
us to set $a=0$ in our weighted Sobolev spaces $W^{s,p}_a =
W^{s,p}_a(\RR^N)$.  This assumption will be removed later on.
The following result is known, we sketch a proof for
completeness. (See for example \cite{Pazy} and \cite{MN} in the more
general case of manifolds with bounded geometry.)

\begin{corollary}\label{cor.st2}
Let $t>0$. There exist constants $C_{r,s} >0$ such that, for any $L
\in \bL_\gamma$ with $L=L^\ast$:
\begin{enumerate}[(i)]
\item $\| e^{tL}f \|_{W^{r,p}(\RR^N)} \le C_{r, s} t^{ (s-r)/2 }
\|f\|_{W^{s,p}(\RR^N)}$, $r \ge s$ real. \label{it.cor.st2.1}
\item There exists $\cG_t^L(x,y) \in \cC^\infty((0,\infty) \times
\bR^N \times \bR^N)$ such that
\begin{align}
e^{tL} f(x) = \int_{\bR^N} \cG^L_t(x,y) f(y)dy.
\end{align}  \label{it.cor.st2.2}
\end{enumerate}
\end{corollary}

\begin{proof}
The part (\ref{it.cor.st2.1}) can be proved using resolvent estimates
and a scaling-in-time argument.  Part (\ref{it.cor.st2.2}) folllows
from the Schwartz kernel theorem (see for example \cite[Chapter
7]{TayPDEII}), since from (\ref{it.cor.st2.1}) $e^{tL}$ maps compactly
supported distributions in $\mathcal{E}'(\RR^N)$ to smooth functions
in $C^\infty(\RR^N)$. In fact, if we denote by $<,>$ the duality
pairing between $C^\infty$ and $\mathcal{E}'$, we explicitly have:
\begin{equation} \label{eq.kerneldef}
      \cG^L_t(x,y)=<\delta_x,e^{tL} \delta_y>,
\end{equation}
where $\delta_z$, $z \in \RR^N$, represents the Dirac delta
distributions supported at $z$ (\ie $\delta_z(f) = f(z)$).
\end{proof}

We now proceed to eliminate the assumption that $L^* = L$ in the above
result.  First, let us notice that if $L, L_0 \in \bL_\gamma$, and if
we denote $V= L - L_0$ and $g(t, x) = Vu(x, t)$, then \eqref{eq.IVP}
becomes
\begin{equation}\label{eq.inhomogeneous-IVP}
\begin{cases}
  \D_t u - L_0 u = g & \hspace{1.0cm} \mbox{in}\; (0,\infty)\times \bR^N,\\
  u(0, x) = f (x) & \hspace{1.0cm} \mbox{on}\; \{0\}\times \bR^N.
\end{cases}
\end{equation}
It is well-know that applying Duhamel's formula, gives a Volterra
integral equation of the first kind for $u$. If $L=L_0^\ast$, the
solution of the integral equation is a classical solution of
\eqref{eq.inhomogeneous-IVP}. in fact, it is enough that $e^{tL_0}$
generates an analytic semigroup (see \cite[Theorem 2.4, page
107]{Pazy}). For simplicity, we want to avoid using the theory of
analytic semigroups, and rather use instead Corrolary \ref{cor.st2}.

\begin{lemma} \label{lemma.Duhamel}
Let us assume that $g \in \cC([0,\infty), L^p)$. Then the classical
$L^p$-solution of the problem \eqref{eq.inhomogeneous-IVP} is given by
\begin{equation} \label{eq.Duhamel0}
    u(t) = e^{t L_0}f + \int_0^t e^{ ( t  -\tau)L_0}g(\tau) d\tau.
\end{equation}
Assume that $L\in \bL_\gamma$, and let $L_0=(L^\ast + L)/2$.  Then,
the classical $L^p$-solution $u(t) =: e^{tL}f$ to the problem
\eqref{eq.IVP0} is given by:
\begin{equation} \label{eq.Duhamel}
    e^{t L}f = e^{t L_0}f + \int_0^t e^{ (  t -\tau)L_0}(L -
    L_0)e^{ \tau L}f d\tau,
\end{equation}
for any initial data $f \in L^{p}$, $1<p<\infty$.
\end{lemma}

\begin{proof}
Let us notice that $u$ is defined since $e^{(t-\tau)L_0}g(\tau)$ is
continuous in $\tau$.  For $f \in W^{2, p}$ and $g \in \cC([0,\infty),
W^{2, p})$, the function $u$ is also differentiable and
\begin{equation*}
  u'(t) = L_0 e^{t L_0}f + \int_0^t L_0 e^{(t-\tau)L_0}g(\tau) d\tau +
  g(t) = L_0u(t) + g(t).
\end{equation*}
Since $\cC([0,\infty), W^{2, p})$ is dense in $\cC([0,\infty), L^p)$
(for the topology of uniform convergence on compacta), this proves the
first part.

To prove the second part, let us chose $f \in W^{1,p}$, then the
function $[0, \infty) \ni \tau \to (L - L_0)e^{\tau L}f =g(\tau) \in
L^p$ is continuous (since $L-L_0$ is in $\bL$ and has order at most
one), and hence we can apply the first part to obtain the formula
\eqref{eq.Duhamel}.

In general, using Proposition \ref{prop.st2}, part
\eqref{it.prop.st2.ii}, and the fact that $V\in \bL$ and has order at
most one, we obtain by Corollary \ref{cor.st2}, part
\eqref{it.cor.st2.1} that \ $\|e^{(t-\tau)L_0} (L-L_0) e^{\tau
L}f\|_{L^p} \le C t^{-1/2} \|f\|_{L^{p}}$, so that the integral on the
right hand side of \eqref{eq.Duhamel} is defined and continuous in $f
\in L^p$.  Since the left hand side of \eqref{eq.Duhamel} is also
continuous in $f \in L^p$, the result then follows by continuity and
by density of $W^{1,p}$ in $L^p$.
\end{proof}

\begin{remark}
The integral $\int_0^t e^{ ( t -\tau)L_0}(L - L_0)e^{ \tau L}f d\tau$
is defined either as a Bochner integral (for the definition of the
Bochner integral see e.g. \cite{BochnerBook}) or as the limit
$\lim_{\epsilon \searrow 0}\int_{ \epsilon}^{t -\epsilon}
e^{(t-\tau)L_0}(L - L_0)e^{\tau L}f d\tau$ of Riemann integral for
continuous functions.
\end{remark}

We now extend Corollary \ref{cor.st2} to non self-adjoint operators
$L$ and to the exponentially weighted spaces $W^{s,p}_a$.

\begin{proposition} \label{prop.mapgeneral}
Let $L \in \bL_\gamma$ arbitrary. We have $e^{tL}W^{s,p}_a \subset
W^{r,p}_a$ for all $r, s, a \in \RR$, $1<p<\infty$,  and $t>0$.
 Let $r \ge s$, then
\begin{equation*}
    \| u(t) \|_{W_a^{r,p}}  \le C  t^{ (s-r)/2 } \|f\|_{W_a^{s,p}},
    \quad t \in (0, 1].
\end{equation*}
The constant $C$ above is independent $r$, $s$, $a$, $p$, and $L$, as
long as they belong to bounded sets.
\end{proposition}

We recall that $W^{s,p}_a$ is independent of the choice of the point
$z$ (see Equation \eqref{def.w.S}). The constant $C$ in the above
proposition is also independent of $z\in \RR^N$ since the family
$e^{a\<x\>_{z}} L e^{-a\<x\>_{z}}$ is uniformly bounded in
$\bL_\gamma$ for $z \in \RR^N$ and $a$ in a bounded set. For this
reason, we shall sometimes drop the index $z$ from the notation
$\<x\>_{z}$.

\begin{proof}
As discussed above, we may assume that $a=0$.  Also, note that we
already know that $e^{tL}W^{s,p}_a \subset W^{r,p}_a$ for all $r \le
s$, so let us concentrate on the non-trivial case $r \ge s$.  Let $L_0
:= (L + L^*)/2$ and $V = L - L_0$. Then $V \in \bL$ is a differential
operator of order at most one. By Lemma \ref{lemma.Duhamel},
\begin{equation*}
    e^{tL} = e^{tL_0} + \int_0^t e^{(t-\tau)L_0} V e^{\tau L}d\tau.
\end{equation*}
Let us assume also that $s \le r < s+1$. Using also Proposition
\ref{prop.st2}, part \ref{it.prop.st2.ii}, we obtain that the norm
$\|e^{tL}\|_{W^{s,p} \to W^{r,p}}$ of $e^{tL}$ as linear map $W^{s,p}
\to W^{r,p}$ can be bounded as
\begin{multline*}
  \|e^{tL}\|_{W^{s,p} \to W^{r,p}} \le
  \|e^{tL_0}\|_{W^{s,p} \to W^{r,p}} \\ + \int_0^t
  \|e^{(t-\tau)L_0}\|_{W^{s-1,p} \to W^{r,p}} \|V\|_{W^{s,p} \to
  W^{s-1,p}} \|e^{\tau L_1}\|_{W^{s,p} } d\tau \\ \le C \Big
  ( t^{(s-r)/2} + \int_0^t (t-\tau)^{(s-1-r)/2}d\tau \Big ) =
  Ct^{(s-r)/2}(1 + t^{1/2}) \le Ct^{(r-s)/2},
\end{multline*}
where in the last inequality we have used that $0<t\le 1$, and where
$C>0$ is a generic constant, different at each appearance.  The
general case follows from this one as follows. Let $\delta = (r-s)/m$, for
$m > r-s$. We first notice that
$\|e^{tL/m}\|_{W^{s + (j-1)\delta,p} \to
W^{s + j\delta,p}} \le C (t/m)^{(s-r)/(2m)}$ by the result that we have
just proved, since $\delta < 1$.
We then write $e^{tL} = \big( e^{tL/m}
\big)^m$ and we use the submultiplicative property of the norm to obtain
$\|e^{tL}\|_{W^{s,p} \to W^{r,p}} \le C (t/m)^{m(s-r)/(2m)} = C' t^{(s-r)/2}$.
\end{proof}

In particular, Proposition \ref{prop.mapgeneral} gives the existence
of the Green's function $\cG^{L}_t(x, y)$ for any $L \in \bL_\gamma$,
defined again via formula \eqref{eq.kerneldef}.  In particular for
$t>0$, this kernel is a smooth function of $x$ and $y$. We will also
use the notation $\cG^{L}_t(x, y) = e^{tL}(x, y)$. The following
corollary is a consequence of Proposition \ref{prop.mapgeneral}.

\begin{corollary}\label{cor.cont}\
Let $L \in \bL_\gamma$ and $s,r\in \RR$ be arbitrary. We then have
that the map
\begin{equation*}
    (0, \infty) \ni t \to e^{tL} \in \cB(W^{s, p}_a,W^{r, p}_a)
\end{equation*}
is infinitely many times differentiable.
\end{corollary}

\begin{proof} We have $\pa_t^k e^{tL} = e^{tL} L^k$, so it is enough
to show that the map $(0, \infty) \ni t \to e^{tL} \in \cB(W^{s-k,
p}_a,W^{r, p}_a)$ is continuous. Now, for each $\delta>0$, let $t \ge
\delta > 0$. Then $e^{\delta L}$ maps $W^{s-k, p}_a$ to $W^{r+2, p}_a$
continuously, by Proposition \ref{prop.mapgeneral}.  Writing $e^{tL} =
e^{(t-\delta)L} e^{\delta L}$ and using the continuity of $[\delta,
\infty) \ni t \to e^{(t-\delta)L} \in \cB(W^{r+2, p}_a, W^{r, p}_a)$,
by Lemma \ref{lemma.cont}, we obtain the result.
\end{proof}

See \cite{Carmona, Hundertmark, Simon} for more continuity properties of the semigroups
generated by second order differential operators (Schr\"odinger semigroups).

Let us notice for further reference that for constant coefficient operators, the Green's
function can be determined explicitly.

\begin{remark} \label{rem.explicit.Kernel}
If $L$ \eqref{eq.L} is a constant coefficient operator
\begin{align}
  L^0 = \sum_{i,j=1}^n a^0_{ij} \D_i \D_j + \sum_{k=1}^n b^0_k\D_k +
  c^0
\end{align}
and $A^0:= (a^0_{ij})$ is the matrix of highest order coefficients,
assumed to satisfy $a^0_{ij} = a^0_{ji}$, we have the explicit formula
\begin{equation} \label{eq.explicit.Kernel}
  \cG^{L^0}_t(x,y) = e^{tL^0}(x, y)=\frac{e^{c^0t}}{\sqrt{(4\pi t)^n
  \det(A^0)}} e^{\frac{(x+b^0t -y)^t (A^0)^{-1} (x +b^0t
  -y)}{4t}}.
\end{equation}
\end{remark}

\subsection{Perturbative expansion} \label{ssec.perturbative}
The purpose of this section is to obtain a time-ordered perturbative
expansion of $e^{tL}$, $L\in \bL_\gamma$, in terms of $e^{tL_0}$ for a
fixed element $L_0\in \bL_\gamma$. Later, $L_0$ will be obtained by
freezing the highest-order coefficients of $L$ at a given point $z$
and dropping the lower-order terms. This expansion is the well-known
Dyson series \cite{JaffeBook, JoachainBook, KatoBook}.  Here, we concentrate on
justifying this expansion in our setting and in obtaining global error
estimates in weighted Sobolev spaces.

For each $k\in \ZZ_+$, we denote by
\begin{multline*}
    \Sigma_k :=
    \{\tau = (\tau_0, \tau_1, \ldots, \tau_{k}) \in \RR^{k+1},\
    \tau_j \ge 0, \sum \tau_j=1\} \\ \simeq
    \{\s=(\s_1, \ldots, \s_k) \in \RR^{k},\
    1 \ge \s_1 \ge \s_2 \ge \ldots \s_{k-1} \ge \s_{k} \ge 0\}
\end{multline*}
the {\em standard unit simplex} of dimension $k$. The identification
above is given by $\s_j = \tau_j+\tau_{j+1}+\ldots+\tau_{k}$.  Using
this bijection, for any operator-valued function $f$ of $\RR^N$ we can
write
\begin{multline*}
    \int_{\Sigma_k}\!\! f(\tau)d\tau
    = \int_{0}^{1}\!\!\int_{0}^{\s_1}\!\!\! \ldots \!\! \int_{0}^{\s_{k-1}}\!\!
    f(1-\s_1, \s_1- \s_2, \ldots, \s_{k-1} - \s_{k}, \s_k)d\s_{k} \ldots d\s_1\\
    = \int_{\Sigma_k} \!\!
    f(1-\s_1, \s_1- \s_2, \ldots, \s_{k-1} - \s_{k}, \s_k)d\s
\end{multline*}

We recall that, if $g :[a,b] \to X$ is a {\em continuous} function to
a Banach space $X$, $\int_a^b g(t)dt$ is defined as a Riemann
integral.  We begin with a preliminary lemma.  We further recall that
$\cB(X, Y)$ the Banach space of continuous, linear maps between two
Banach spaces $X$ and $Y$.

\begin{lemma}\label{lemma.perturbative}
Let $L_j \in \bL_\gamma$ and let $V_j$ be such that
$e^{-b_j\<x\>} V_j\in \bL$,
$j=1,\ldots,k$, for some $b=(b_1,\ldots,b_k)\in
\RR_+^k$, $k \in \ZZ_+$. Then
\begin{equation*}
    \Phi(\tau) = e^{\tau_0L_0}V_1 e^{\tau_1L_1} \dots e^{\tau_{k-1} L_{k-1}}
    V_k e^{\tau_kL_k}, \qquad \tau \in \Sigma_k
\end{equation*}
defines a continuous function
$\Phi : \Sigma_k \to \cB(W^{s, p}_a(\RR^N), W^{r,p}_ {a- |b|}(\RR^N))$
for any $a \in \RR$  and $1<p<\infty$.
\end{lemma}

Above  we use the standard multi index notation $\, |b|=\sum_{j=1}^k b_j$.

\begin{proof}
It is enough to prove that $\Phi$ is continuous on each of the sets $\cV_j :=\{
\tau_j > 1/(k+2)\}$, $j=0,\dots,k$, since they cover $\Sigma_k$.
Let us assume that $j=0$, for the simplicity of notation.

By assumption and by Lemma \ref{lemma.cont}, each of the
functions
\begin{equation*}
    [0, \infty) \ni \tau_j \to V_j e^{\tau_jL_j} \in
    \cB(W^{r_j+4, p}_{c_j},W^{r_j, p}_{c_j-b_j}), \quad 1 \le j \le k,
\end{equation*}
is continuous. For a suitable choice of $c_j$ and $r_j$ (more precisely,
$c_j = c_{j+1} -b_{j+1}$, $c_k = a$, $r_j = r_{j+1} - 4$, $r_k = s$), we
obtain that the map
\begin{equation*}
    [0, \infty)^k \ni (\tau_j)=: \tau' \to \Psi(\tau') :=
    V_1 e^{\tau_1L_1} ... V_k e^{\tau_kL_k}\in
    \cB(W^{s, p}_{a},W^{s-4k, p}_{a-|b|})
\end{equation*}
is continuous.

Corollary \ref{cor.cont} gives that the map $\tau_0 \to e^{\tau_0L_0}
\in \cB(W^{s-4k, p}_{a-|b|},W^{r, p}_{a-|b|})$ is continuous for
$\tau_0 \ge 1/(k+2)$. This proves the continuity of $\Phi$ on
$\cV_0$ and completes the proof of the lemma.
\end{proof}

By iterating Duhamel's formula in Lemma \ref{lemma.Duhamel},
we obtain a time-ordered expansion of $e^{tL}$.

\begin{proposition} \label{prop.perturbative}
Let $d\in \ZZ_+$. Then, for each $L,L_0\in \bL_\gamma$,
\begin{align}\label{eq.perturbative}
\begin{split}
&e^{tL} =  \, e^{tL_0} + t\int_{\Sigma_1} e^{t \tau_0L_0}Ve^{ t\tau_1L_0} d\tau
\\
& + t^2 \int_{\Sigma_2} e^{ t \tau_0L_0}Ve^{t \tau_1L_0}Ve^{ t \tau_2L_0} d\tau
 +  \dots +  \\
& +  t^d \int_{\Sigma_{p}}e^{t \tau_0L_0}Ve^{t \tau_1L_0} \dots
e^{t \tau_{d-1}L_0}Ve^{t\tau_{d} L_0}d\tau \\
 &+ t^{d+1} \int_{\Sigma_{d+1}}e^{t\tau_0L_0}Ve^{t\tau_1L_0} \dots
e^{t\tau_{d}L_0}Ve^{\tau_{d+1} L}d\tau,
\end{split}
\end{align}
where $V=L-L_0$, and each integral is a well-defined Riemann integral
of a Banach valued function.
\end{proposition}

The positive integer $d$ will be called the {\em iteration level} of
the approximation.  As $d\to\infty$, formula \eqref{eq.perturbative}
above gives rise to an asymptotic series ({\em Dyson series}, see
\cite{JaffeBook, JoachainBook, KatoBook} and the references therein).

Later in the paper, $V$ will be replaced by a Taylor approximation of
$L$, so that $V$ will have polynomial coefficients in $x$, so we have
included this case in the lemma above.

\begin{proof}
Recall that Lemma \ref{lemma.Duhamel} gives
\begin{equation*}
  e^{t L} - e^{t L_0} =\int_0^t e^{(1-\zeta) L_0} V e^{\zeta L}d\zeta
  =\int_0^1 e^{t(1-\tau)L_0} V e^{ t \tau L} t d\tau.
\end{equation*}
with the substitution $\zeta=t \tau$. This is in fact our result for $k=1$.

The result for any $p$ then follows by induction using the above
formula.

Recall that on each simplex $\Sigma_p$, we denoted $\sigma_k = \tau_k
+ \tau_{k+1} + \ldots + \tau_{p}$.  Explicitly, for $t=1$ we have
\begin{multline*}
e^{L} = e^{L_0} + \int_{\Sigma_1} e^{(1-\s_1)L_0}Ve^{\s_1L_0} d\s+
\int_{\Sigma_2} e^{(1-\s_1)L_0}Ve^{(\s_1-\s_2)L_0}Ve^{\s_2L_0} d\s \\
+ \dots + \int_{\Sigma_{d-1}}e^{(1-\s_1)L_0}V \dots
V e^{(\s_{d-2}-\s_{d-1})L_0}Ve^{\s_{d-1}L}d\s\\ = e^{L_0} +
\int_{\Sigma_1} e^{(1-\s_1)L_0}Ve^{\s_1L_0} d\s+ \int_{\Sigma_2}
e^{(1-\s_1)L_0}Ve^{(\s_1-\s_2)L_0}Ve^{\s_2L_0} d\s \\ + \dots +
\int_{\Sigma_{d-1}}e^{(1-\s_1)L_0}V \dots V
e^{(\s_{d-2}-\s_{d-1})L_0}Ve^{\s_{d-1}L_0}d\s + \dots \\
+\int_{\Sigma_{d-1}}\int_{0}^{\s_{d-1}} e^{(1-\s_1)L_0}V \dots
e^{(\s_{d-2}-\s_{d-1})L_0}Ve^{(\s_{d-1} - \s_d) L_0} V e^{\s_d L}d\s
d\s_n \\ = e^{L_0} + \int_{\Sigma_1} e^{(1-\s_1)L_0}Ve^{\s_1L_0} d\s+
\int_{\Sigma_2} e^{(1-\s_1)L_0}Ve^{(\s_1-\s_2)L_0}Ve^{\s_2L_0} d\s \\
+ \dots + \int_{\Sigma_d}e^{(1-\s_1)L_0}Ve^{(\s_1-\s_2)L_0} \dots
e^{(\s_{d-1}-\s_d)L_0}Ve^{\s_d L}d\s,
\end{multline*}
where each integral is well defined as a Riemann integral by the Lemma
\ref{lemma.perturbative}.
\end{proof}

\section{Local dilations and perturbative expansions}
\label{sec.dilations}

In this section, we tackle the task of deriving an algorithmically
computable approximation to $e^{tL}$. We exploit the perturbative
expansion \eqref{eq.perturbative} with $L_0$ the operator obtained by
freezing the highest-order coefficents of $L$ at a given, but
arbitrary, point $z\in \RR^N$, and dropping the lower-order terms (see
\eqref{eq.L_0} below).  Then, we approximate $L-L_0$ by an appropriate
Taylor expansion, so that each of the terms in \eqref{eq.perturbative}
except the last one can be explicitly computed using commutator
formulas, as discussed in Section \ref{sec.commutator}.  Recall that
the sets of second order differential operators $\bL_\gamma \subset
\bL$ were introduced in Definition \ref{def.Lgamma}.

First, using a suitable rescaling in space and time, we replace the
problem of determining an asymptotic expansion of the kernel
\begin{equation*}
\cG_t^{L}(x, y) := e^{tL}(x, y)
\end{equation*}
of $e^{tL}$ by the problem of determining an asymptotic expansion of
the kernel $\cG_1^{L^{s,z}}(x, y) = e^{L^{s,z}}(x, y)$ of
$e^{L^{s,z}}$ for a suitable family of operators $L^{s,z}$
parameterized by $s = \sqrt{t}$, and by the point $z \in \RR^N$. The
point $z$ is fixed throughout this section, but it will be allowed to
vary later on as a function of $x$ and $y$ satisfying some conditions,
for example $z=(x+y)/2$. For some results, we will set $z=x$.  The
family $L^{s,z}$ has limit precisely $L_0$ as $s \to 0$. Since we will
let $z$ vary later, we shall sometimes write $L_0=L_0^{z}$.

For any $s>0$, we consider the action on functions of dilating $x$ by
$s$ about $z$ and $t$ by $s^2$ about $0$. If $f : \RR^N \to \RR$,
$u:[0,\infty)\times \RR^N\to \RR$, we then set
\begin{align}\label{eq.dilation}
    f^{s,z}(x) := f(z + s(x-z)),\\u^{s,z}(t,x) := u(s^2t, z + s(x-z)),
\end{align}
and,
\begin{equation} \label{eq.L^sdef}
    L^{s,z} := \sum_{i,j=1}^N a^{s,z}_{ij}(x) \D_i \D_j + s \sum_{i=1}^N
b^{s,z}_i(x) \D_i + s^2 c^{s,z}(x).
\end{equation}
We immediately see that
\begin{align}\label{eq.L^s}
L^{s,z} u^{s,z} = s^2(Lu)^{s,z}, \;\;\; (\D_t -  L^{s,z})u^{s,z} = s^2[(\D_t - L)u]^{s,z}
\end{align}
In particular, we have the following simple lemma, which we record for
further reference.

\begin{lemma}\label{lemma.dilation}
If $u$ solves \eqref{eq.inhomogeneous-IVP}, then $u^{s,z}$ solves
\begin{equation}\label{eq.dilation-inhomogeneus-IVP}
\begin{cases}
  \D_t u^{s,z} - L^{s,z} u^{s,z} = 0 & \hspace{1.0cm} \text{ in }\;
  (0,\infty)\times \bR^N\\ u^{s,z} =f^{s,z} \in
  \cC_c^\infty(\bR^N)& \hspace{1.0cm} \text{ on }\; \{0\}\times \bR^N.
\end{cases}
\end{equation}
\end{lemma}

\subsection{Dilations and Green's functions}
\label{subsec.dilations}
We want to study the Initial Value Problem
\eqref{eq.dilation-inhomogeneus-IVP} and the Green's function of its
associated solution operator $e^{t L^{s,z}}$.  We can reduce to study
the special case $z=0$.

The definition of the Green's function and Lemma \ref{lemma.dilation}
then gives
\begin{align}
\begin{split}
  u^{s,0}(t,x) &= \int_{\bR^N} \cG_t^{L^{s,0}}(x,y) f^{s,z}(y)dy =
  \int_{\bR^N} \cG_t^{ L^{s,0}}(x,y)f(sy)dy \\ & = s^{-N} \int_{\bR^N}
  \cG_t^{L^{s,0}}(x,\frac{y}{s}) f(y) dy.
\end{split}
\end{align}
On the other hand,
\begin{equation*}
  u^{s,0}(t,x) = u(s^2t,sx) = \int_{\bR^N}\cG_{s^2t}^{L}(sx,y) f(y)dy,
\end{equation*}
which implies
\begin{equation*}
\cG_t^{L^{s,0}}(x,\frac{y}{s}) = s^N \cG_{s^2t}^{L}(sx,y) \
\Leftrightarrow\ \cG_t^{L^{s,0}}(x,y) = s^N \cG_{s^2t}^{L}(sx,sy).
\end{equation*}
In other words
\begin{equation*}
\cG_t^{ L}(x,y) = s^{-N} \cG^{L^{s,0}}_{s^{-2}t}(s^{-1}x, s^{-1}y)
\end{equation*}

If we now translate to $z\ne 0$ and choose $s=\sqrt{t}$, we obtain the
desired correspondence between $\cG^L_t$ and $\cG^{L^{s,z}}_1$, which
we also record for further reference.

\begin{lemma}\label{lemma.Green.funct.cor}
Assume $L \in \bL$ and let $z$ be a fixed, but arbitrary, point in
$\RR^N$. Then, for any $s>0$,
\begin{equation*}
\begin{aligned}
  \cG_t^{L}(x,y) &= s^{-N} \cG^{L^{s,z}}_{1}(z + s^{-1}(x - z), z +
  s^{-1}(y-z))\\ &=t^{-\frac{N}{2}}\cG_1^{L^{\sqrt{t}, z }}(z +
  t^{-\frac{1}{2}}(x - z), z + t^{-\frac{1}{2}}(y-z)), \text{ if } s =
  t^{-\frac{1}{2}}.
\end{aligned}
\end{equation*}
\end{lemma}

\subsection{Perturbative expansion of $e^{L^{s,z}}$}
Since Lemma \ref{lemma.Green.funct.cor} gives us an
immediate procedure for obtaining the Green function $\cG_t^{L}(x,y)$
of $\partial_t - L$ from the Green's function $\cG_t^{L^{s,z}}(x,y)$
of $\partial_t - L^{s,z}$, we now concentrate on obtaining a
perturbative expansion for the latter.

Recall that $L_0^{z} = L^{0, z} = \lim_{s \searrow 0} L^{s, z}$.
Let us write $V_1^{s,z} := L^{s,z} - L_0^{z}$. Then, $V_1^{s,z}$
takes the role of $V$ in the perturbative expansion
\eqref{eq.perturbative} for the operator $e^{L^{s,z}}$, that is:
\begin{align}\label{eq.perturbative1}
\begin{split}
&e^{L^{s,z}} = \, e^{L_0^{z}} + \int_{\Sigma_1}
e^{\tau_0L_0^{z}}V^{s,z}_1e^{\tau_1L_0^{z}} d\tau \\
&  +
\int_{\Sigma_2}
e^{\tau_0L_0^{z}}V^{s,z}_1e^{\tau_1L_0^{z}}
V^{s,z}_1e^{\tau_2L_0^{z}} d\tau \, + \dots
\\ & \qquad \quad
+ \int_{\Sigma_{d}}e^{\tau_0L_0^{z}}V^{s,z}_1e^{\tau_1L_0^{z}}
\ldots e^{\tau_{d-1}L_0^{z}}V^{s,z}_1e^{\tau_{d} L_0^{z}}d\tau \,
\\ &\qquad \qquad \qquad
+ \int_{\Sigma_{d+1}}e^{\tau_0L_0^{z}}V^{s,z}_1e^{\tau_1L_0^{z}}
\ldots e^{\tau_{d}L_0^{z}}V^{s,z}_1e^{\tau_{d+1} L^{s,z}}d\tau.
\end{split}
\end{align}
In a sense to be made precise below, we have $V_1^s = \cO(s)$. Consequently,
if we let the iteration level $d\to \infty$ in \eqref{eq.perturbative},
we obtain a formal power series in $s$. We will rigorously show in
Section \ref{sec.EE} using the exponentially weighted Sobolev spaces
$W^{s,p}_a$ that \eqref{eq.perturbative} indeed gives rises to an
asymptotically convergent series in $s$ as $s\to 0$ and will derive
global error bounds in $W^{s,p}$ and $W^{s,p}_a$ for the partial sums.

Let $n \in \bZ_+$ be a fixed integer and consider the Taylor expansion
of the operator $L^{s,z}$ up to order $n$ in $s$ around $s=0$,
\begin{equation} \label{eq.taylorexp}
L^{s,z} = \sum_{m=0}^n s^m L^{z}_m + V_{n+1}^{s,z}
\end{equation}
were $V_{n+1}^{s,z}$ is the remainder term in the expansion. Let
\[
        V^{s,z}_{n+1} = s^{n+1} L^{s,z}_{n+1}.
\]
The operators $L^{z}_m$, $1\leq m\leq n$, are given by
\beq\label{eq.L_m} L^{z}_m := \left. \frac{1}{m!}\left(
\frac{d^m}{ds^m} L^{s,z} \right) \right|_{s=0}, \eeq and are
independent of $s$, while \beq \label{eq.L_n+1} L^{s,z}_{n+1} :=
\left. \frac{1}{(n+1)!}\left( \frac{d^{n+1}}{d\theta^{n+1}}
L^{\theta,z} \right) \right|_{\theta=\alpha s}, \eeq for some
$0<\alpha<1$, and hence it still depends on $s$.

\begin{remark} \label{rem.pertexp}
{From} the form of $L^{s,z}$ in equation \eqref{eq.L^sdef} it follows
that the operator $L_m^{z}$, $m \le n$, (respectively $L^{s,z}_{n+1}$)
has coefficients that are {\em polynomials in $x-z$ of degree at most
$m$} (respectively of degree $n+1$). The coefficients of the
polynomials themselves are bounded functions of $z$.  More precisely,
the coefficients of the second order derivative terms are of degree at
most $m$ in $x-z$, while the coefficients of the first order
derivatives term are of degree at most $m-1$ in $x-z$, and the
coefficients of the zero order derivative term is of degree at most
$m-2$ in $x-z$. The coefficients of these polynomials in $x-z$ are
{\em bounded} functions of $z$, together will all their derivatives, a fact
that will be exploited later.

The first few terms of the Taylor expansions are explicitly:
\begin{subequations}\label{eq.L_mExp}
\beq\label{eq.L_0} L_0^{z} = \sum_{i,j=1}^N a_{ij}(z) \D_i \D_j,
\eeq \beq \label{eq.L_1} L_1^{z} = \sum_{i,j=1}^N \left((x-z)\cdot
\nabla a_{ij}(z) \right)\D_i \D_j + \sum_{i=1}^N b_i(z) \D_i, \eeq
\beq \label{eq.L_2}
\begin{aligned}
L_2^{z} = \sum_{i,j=1}^N \frac{1}{2} \left(
(x-z)^{T}\right. &\left.\nabla^2 a_{ij}(z)\ (x-z) \right)\D_i
\D_j + \\\ &+ \sum_{i=1}^N ((x-z)\cdot \nabla b_i(z)) \D_i +
c(z).
\end{aligned}
\eeq
\end{subequations}
\end{remark}

Since $L_0^{z}$ has coefficients that are constant in $x$, from
formula \eqref{eq.explicit.Kernel} we obtain
\begin{equation} \label{eq.Lokern}
    e^{tL_0^{z}}=  \frac{1}{\sqrt{(4\pi t)^N \det{A^0}}}
        e^{ \frac{(x-y)^t (A^0)^{-1} (x-y)}{4t}},
\end{equation}
where $A^0:=A(z)$.

Furthermore $V^s_1 := L^{s,z} - L_0^{z}$ can be written as
\begin{equation}\label{eq.expansion}
    V^s_1 := \sum_{m=0}^n s^m L^{z}_m + s^{n+1} L_{n+1}^{s,z}.
\end{equation}
This Taylor polynomial expansion can then be substituted into \eqref{eq.perturbative1},
yielding another polynomial in
$s$. To describe each term of this polynomial and to
formulate the main results in this section, we need to introduce some
notation. Let $\NN := \{1, 2, 3, \ldots \}$ denote the set of natural
numbers (always assumed to be $>0$).

\begin{definition}\label{def.mfkA}
For any integers $1 \le k \le d+1$ and $\ell$, we shall denote by
$\fA_{k,\ell}$ the set of multi-indexes $\alpha = (\alpha_1, \alpha_2,
\ldots, \alpha_k) \in \NN^k$, such that $|\alpha| := \sum \alpha_j =
\ell$. Furthermore, we denote $\fA_{\ell} := \bigcup_{k=1}^\ell
\fA_{k,\ell}$. For symmetry, it will be convenient to set $\fA_{\ell,
k} = \{\emptyset\}$ if $\ell < k$, including when $\ell\leq 0$.
\end{definition}

We note that, since $\alpha_i\geq 1$, the set $\fA_{k,\ell}$ is empty
if $\ell<k$.  The meaning of $\ell$ is that of the corresponding power
of $s$ and the meaning of $k$ is that of the iteration level in the
Dyson series \eqref{eq.perturbative1}.

\begin{proposition}\label{prop.size.A.ell}
The set $\fA_{\ell}$ contains $2^{\ell-1}$ elements.
\end{proposition}

\begin{proof}
For any given, $1\leq k\leq \ell$, the number of elements in the set
$\fA_{k,\ell}$ is the number of sequences $\{\alpha_1, \alpha_2,
... \alpha_k \}$ of size $k$ which add up to $\ell$ and is, therefore,
given by ${\ell-1 \choose k-1}$. Consequently, the number of elements
in $\fA_\ell$ is given by $\sum_{k=1}^\ell {\ell-1 \choose k-1} =
\sum_{k=0}^{\ell-1} {\ell-1 \choose k} = 2^{\ell-1}$.
\end{proof}

We are now in the position to describe the expansion
\ref{eq.perturbative1} more explicitly. We recall that $d$ is the
iteration level of the approximation and $n$ is the order of the
Taylor expansion. In the following definition, by abuse of notation,
it will be convenient to write $L_{n+1}^{z}$ instead of
$L_{n+1}^{s,z}$, that is, we shall omit $s$ from the notation. We also
recall that $\fA_{k,\ell}\equiv \emptyset$, if $\ell<k$. This
condition will be understood.

\begin{definition}
For each multi-index $\alpha = (\alpha_1, \ldots, \alpha_k) \in
\fA_{k,\ell}$, we let
\begin{equation} \label{eq.Lambda_alpha}
    \Lambda_{\alpha,z} := \int_{\Sigma_k} e^{\tau_0 L_0^{z}}
    L^{z}_{\alpha_1} e^{\tau_1 L_0^{z}} L^{z}_{\alpha_2} \cdots
    L^{z}_{\alpha_k} e^{\tau_k L_0^{z}} d\tau,
\end{equation}
if $1\leq k\leq d$, and
\begin{equation} \label{eq.Lambda_n}
    \Lambda_{\alpha, z} := \int_{\Sigma_{d+1}} e^{\tau_0 L_0^{z}}
    L^{z}_{\alpha_1} e^{\tau_1 L_0^{z}} L^{z}_{\alpha_2} \cdots
    L^{z}_{\alpha_{d+1}} e^{\tau_{d+1} L^{s,z}} d\tau.
\end{equation}
if $k=d+1$.
Then, we set
\begin{equation} \label{eq.Lambda_ell.int}
    \Lambda^{\ell}_{z} := \sum_{\alpha \in \fA_{\ell} }
    \Lambda_{\alpha,z},
\end{equation}
with the convention that $\Lambda^{0}_z = e^{L_0^{z}}$.
\end{definition}

We observe that $\alpha$ uniquely determines $k$ and $\ell$, so that
our notation is justified.  Let $\alpha = (\alpha_j) \in \fA_{k,
\ell}$.  We remark that if $k= n+1$ or some $\alpha_j=n+1$ (in which
case $L^{z}_{\alpha_j}$ stands in fact for $L^{s,z}_{n+1}$), then
$\Lambda_{\alpha,z}$ \ and \ $\Lambda^{\ell}_{z}$ depend on $s$, so we
shall sometimes denote these terms by $\Lambda_{\alpha,s,z}$ \ and \
$\Lambda^{\ell}_{s,z}$.

Also, in what follows, when no confusion can
arise, we will drop the explicit dependence on $z$. However, in
Section \ref{sec.EE}, $z$ will be allowed to vary and we will
reinstate the full notation. We also observe that each
$\Lambda^{\ell}_{z}$ or $\Lambda^\ell_{s,z}$ is well defined as a
Riemann integral by Lemma \ref{lemma.perturbative} and by the
following lemmas.  Let us recall that $ \<x\>_{w} = (1+|x-w|^2)^{1/2}$.

\begin{lemma}\label{lemma.poly.growth}
The family
\begin{equation*}
    \{\<x\>_{z}^{-j} L_j^{z};
     \ s \in (0, 1],\ z \in \RR^N, \ j=0, \ldots, n+1\}
\end{equation*}
defines a bounded subset of $\bL$.
\end{lemma}

\begin{proof}
This is an immediate consequence of Remark
\ref{rem.pertexp} if $j \le n$ and of directly  estimating
the remainder in the Taylor series for $j=n+1$.
\end{proof}

In the following Lemma, we shall use an {\em arbitrary} center
for our weight.

\begin{lemma}\label{cor.exp.growth}
For each given $\epsilon > 0$, the family
\begin{equation*}
   \{
    e^{-\epsilon\<z-w\>}
   e^{-\epsilon \<x\>_{w}} L_j^{z};
     \ s \in (0, 1], \ z \in \RR^N, \ j=0, \ldots, n+1\}
\end{equation*}
is a bounded subset of $\bL$.
\end{lemma}

\begin{proof} Let us assume first that $w=z$. We need to prove that the family
\begin{equation*}
   	\{   e^{-\epsilon \<x\>_{z}} L_j^{z};
    	 \ s \in (0, 1], \ z \in \RR^N, \ j=0, \ldots, n+1\}
\end{equation*}
is  bounded in $\bL$.
Indeed, this follows from Lemma \ref{lemma.poly.growth} and the simple
observation that $\<x\>_{z}^j e^{-\epsilon \<x\>_{z}} \leq C$,
with $C$ independent of $z$ and $j$.

To obtain the statement of the theorem, we then apply the triangle
inequality to the vectors $(0, x), (1,z), (1,w) \in \RR^{1+N}$ to conclude that
$\<x-z\> - \<x-w\> \le |z-w| \le \<z-w\>$. This shows that $e^{\epsilon (\<x-z\> - \<x-w\> - \<z-w\>)} \le 1$.
Hence the family
\begin{equation*}
	\{e^{\epsilon (\<x-z\> - \<x-w\> - \<z-w\>)} e^{-\epsilon \<x\>_{z}}
	L_j^{z} = e^{-\epsilon\<z-w\>} e^{-\epsilon \<x\>_{w}} L_j^{z}\},
\end{equation*}
$s \in (0, 1]$, \ $z \in \RR^N$, \  $j=0, \ldots, n+1$, is bounded in $\bL$, as claimed.
\end{proof}

Lemma \ref{lemma.perturbative} together with Lemma
\ref{cor.exp.growth} then give the following result.

\begin{corollary}\label{cor.wel.def}
We have $\Lambda_{\alpha, z} \in \cB(W^{s, p}_a, W^{r,
  p}_{a-\epsilon})$, for any $\alpha \in \fA_{k, \ell}$, $z \in
\RR^N$, $r, s \in \RR$, $1<p<\infty$, and $\epsilon > 0$. Moreover, we
have that
\begin{equation*}
		\|\Lambda_{\alpha, z}\|_{W^{q, p}_{a, z} \to W^{r, p}_{a-\epsilon, z}} \le
		C_{q, r, p, a, \epsilon}
		e^{k \epsilon \<z-w\>},
\end{equation*}
for a constant $C_{q, r, p, a, \epsilon}$ that does not depend on
$z$. In particular, each $\Lambda_{\alpha, z}$ is an operator with
smooth kernel $\Lambda_{\alpha, z}(x, y)$.
\end{corollary}

Therefore, we can write
\begin{equation}
	\Lambda_{\alpha, z} f(x) = \int_{\RR^N} \Lambda_{\alpha, z}(x, y) f(y) dy.
\end{equation}

The point of the above definition and results is to rewrite the
perturbative expansion (partial Dyson series) in the form

\begin{lemma}\label{lemma.L.pert} Denote $M =  (d+1)(n+1)$. We have
\begin{equation*}
  e^{L^{s,z}}= \ e^{L_0^{z}} + \sum_{\ell = 1}^{M} \sum_
         {k=1}^{\min\{\ell, d+1\}} \sum_{\alpha \in \fA_{k, \ell}}
         s^\ell \Lambda_{\alpha, z} =\sum_{\ell = 0}^{M} s^\ell
         \Lambda_{z}^{\ell}.
\end{equation*}
\end{lemma}

We now assume that $n \le d$ and write the perturbative expansion
of the above Lemma as follows:
\begin{multline} \label{eq.L.pert}
         e^{L^{s,z}}= \ e^{L_0^{z}} + \sum_{\ell = 1}^n s^\ell
         \Lambda^{\ell}_{z} +\sum_{\ell = n+1}^{M} s^\ell
         \Lambda^{\ell}_{z} \\ = \ e^{L_0^{z}} + \sum_{\ell = 1}^n
         s^\ell \Lambda^{\ell}_{z} +s^{n+1} \bE^{s,z}_{d,n} =
         \sum_{\ell = 0}^n s^\ell \Lambda^{\ell}_{z} + s^{n+1}
         \bE^{s,z}_{d,n},
\end{multline}
where $\bE^{s,z}_{d,n}$ represents the error in the approximation and
depends on $s$, whereas the terms $\Lambda^{\ell}_{z}$, $1 \le \ell
\le n$ {\em do not depend} on $s$ or $d$, since we have assumed that
$n \le d$.  Since $\bE^{s, z}_{d, n}$ is independent of $d$ for $d \ge
n$, we shall eventually restrict to $d=n$.

\section{Commutator calculations} \label{sec.commutator}
The purpose of this section is to give an explicitly computable
representation of the perturbative expansion \eqref{eq.L.pert} as
\begin{equation*}
    e^{L^{s,z}} \sim
    e^{L_0^{z}} + \sum_{\ell = 1}^{n} s^\ell
    \cP^{\ell}(x,z,\D)e^{L_0^z}
\end{equation*}
where $\cP^{j}(x,z,\D)$ is a differential operator with smooth
coefficients that depend polynomialy on $x-z$ and $s$, and are bounded
with all derivatives in $z$. Both the order of the operator as well as
the degree of the polynomial coefficients depend on the order of the
Taylor expansion $n$, which also equals the iteration level $d$.  We
give an explicit characterization of $\bP_{n}$ and an iterative
procedure to calculate it in Theorem \ref{thm.rep}.  The main idea is
to show that each $\Lambda_{\alpha, z}$ in \eqref{eq.Lambda_alpha} can
be written as an explicitly computable differential operator
$\cP_\alpha$ acting on the distribution kernel of $e^{L_0^{z}}$, and
thus using \eqref{eq.Lambda_ell.int} show that the perturbative
expansion \eqref{eq.L.pert} can be rewritten in this form as well.
Throughout this section, $z$ is kept {\em fixed}, though arbitrary,
and $\partial$ will always mean differentiation with respect to $x$.

\begin{definition}[Spaces of Differentiatial Operators]
For any nonnegative integers $a,b$ we denote by $\cD(a,b)$ the vector
space of all differentiations of polynomial degree at most $a$ and
order at most $b$.  We extend this definition to negative indices by
defining $\cD(a,b) = \{ 0\}$ if either $a$ or $b$ is negative. By
polynomial degree of $A$ we mean the highest power of the polynomials
appearing as coefficients in $A$.
\end{definition}

We remark that $\cD(0, b)$ consists of differential operators with
{\em constant coefficients}.

\begin{definition}[Adjoint Representation]
For any two differentiations $A_1 \in \cD(a_1,b_1)$ and $A_2 \in
\cD(a_2,b_2)$ we define $\adj_{A_1}(A_2)$ by
\begin{equation}
\adj_{A_1}(A_2) := [A_1, A_2] = A_1 A_2 - A_2 A_1,
\end{equation}
as usual, and for any integer $j \geq 1$ we define $\adj^j_{A_1}(A_2)$
recursively by
\begin{equation}\label{eq.ad.recursive}
\adj^j_{A_1}(A_2) :=\adj_{A_1}(\adj^{j-1}_{A_1}(A_2))
\end{equation}
\end{definition}

\begin{proposition} \label{lemma.order-of-commutators}
Suppose $A_1 \in \cD(a_1,b_1)$ and $A_2 \in \cD(a_2,b_2)$. Then for
any integer $k \geq 1$, \ $\adj_{A_1}^k (A_2) \in \cD(k(a_1 -1) + a_2,
k(b_1 -1) + b_2).$
\end{proposition}

\begin{proof} We first notice that
\begin{align}\label{eq.ad.algebra}
\adj_{A_1}(A_2)  \in \cD(a_1 -1 +a_2, b_1 - 1 + b_2).
\end{align} Next, from \eqref{eq.ad.recursive} we have
\begin{align}
\adj_{A_1}^k(A_2) = \adj_{A_1}(\adj_{A_1} (\adj_{A_1}(\adj_{A_1}(
\dots )))), \quad k-\mbox{times,}
\end{align}
so that an application of \eqref{eq.ad.algebra} $k$ times yields the result.
\end{proof}

\begin{lemma} \label{lemma.higher-order-commutator}
Let $m,k$ be fixed integers $\geq 1$.  Let $L_0\in \cD(0,2)$ and
$L_m \in \cD(m,2)$. Then, $\adj^k_{L_0}(L_m) \in \cD(m-k,
k +2).$
In particular,
\begin{align} \label{eq.p.ad.zero}
    \adj^k_{L_0}(L_m) = 0, \qquad \text{if } k > m.
\end{align}
\end{lemma}

\begin{proof}
Applying Lemma \ref{lemma.higher-order-commutator} we see that ${\rm
  ad}^k_{L_0^{z}}(L^{z}_m) \in \cD(m-k,k+2)$.  If $k>m$, then by
definition $\cD(m-k,k+2) = \{0\}$ and we obtain \eqref{eq.p.ad.zero}.
\end{proof}

\begin{lemma}\label{lemma.finite-BCH}
Let $L_0 \in \cD(0,2) \cap \bL_\gamma$, and let $L_m \in
\cD(m,2)$. Then for any $\theta > 0$,
\begin{equation*}
e^{\theta L_0} L_m =P_m(L_0, L_m; \theta,x,\D) e^{\theta L_0},
\end{equation*}
where $P_m(\theta) =  P_m(L_0, L_m; \theta,x,\D) \in \cD(m ,m+2)
$ is given by
\begin{equation*}
    P_m(\theta) := \sum_{k=0}^m \frac{\theta^k}{k !}\adj^k_{L_0}(L_m)
    = L_m + \theta[L_0, L_m] + \frac{\theta^2}{2}[L_0, [L_0, L_m]] +
    \cdots .
\end{equation*}
\end{lemma}

\begin{proof}
Recall the Baker-Campbell-Hausdorff formula (see for instance \cite{B,C,H})
\begin{align}
    \Phi(t) : = e^{tA} B - \left( \sum_{k=0}^\infty t^k
    \adj_A^k(B)/k! \right) e^{tA} = 0.
\end{align}
In general, this formula is a formal infinite series, and the equality $\Phi(t) = 0$
must be justified.

Setting $A = L_0$, $B=L_m$, we have that $\adj_A^{m+1}(B) = 0$, by
Lemma \ref{lemma.higher-order-commutator}, so the sum becomes finite, and the function
$\Phi(t)$ is well defined as a bounded operator $W^{m, p}_1 \to L^p$. Since $\Phi(0) = 0$,
to prove that $\Phi(t) =0$ for all $t$, it is enough to show that $\pa_t \Phi(t) f = 0$
for all $f \in W^{m, p}_1$. Indeed, we have
\begin{multline*}
    \pa_t \Phi(t) f = e^{tA} ABf \\ - \left( \sum_{k=0}^\infty k
    t^{k-1} \adj_A^k(B)/k! \right) e^{tA}f - \left( \sum_{k=0}^\infty
    t^k \adj_A^k(B)/k! \right) A e^{tA}f \\ = e^{tA} ABf - \left(
    \sum_{k=0}^\infty t^{k} \adj_A^{k+1}(B)/k! \right) e^{tA}f -
    \left( \sum_{k=0}^\infty t^k \adj_A^k(BA)/k! \right) e^{tA}f\\ =
    e^{tA} ABf - \left( \sum_{k=0}^\infty t^{k} \adj_A^{k}(AB)/k!
    \right) e^{tA}f \\ = A e^{tA} Bf - A\left( \sum_{k=0}^\infty t^{k}
    \adj_A^{k}(B)/k! \right) e^{tA}f = A \Phi(t) f.
\end{multline*}
So the continuous function $u(t) := \Phi(t) f \in L^p$ satisfies the
equation $\pa_t u(t) - Au(t) = 0$ with initial condition $u(0) =
0$. By the uniqueness of the solutions of this equation in $L^p$, we
obtain that $u(t) =0$, which is the desired Baker-Campbell-Hausdorff
formula.

The indicated properties of $P_m(\theta) = P_m(L_0, L_m; \theta,x,\D)$
are obtained directly from Lemma \ref{lemma.higher-order-commutator},
as follows. We have $\adj_A^k (B) \in \cD(m-k, k+2)$ and hence
\begin{equation*} \label{eq.formula.Pm}
    P_m(\theta) := \sum_{k=0}^{m} \frac{\theta^k}{k!} \adj_{A}^k(B)
    \in \sum_{k=0}^m \cD(m-k, k+2) \subset \cD(m, m+2).
\end{equation*}
This completes the proof.
\end{proof}

\begin{lemma}\label{lemma.explicit}
For a given multi-index $\alpha \in \fA_{k,\ell}$ with $k \le d = n$,
let
\begin{equation*}
  {\cP}_{\alpha}(x, z, \D) := \int_{\Sigma_k} \prod_{i=1}^k
  P_{\alpha_i} (L_0^{z}, L_{\alpha_i}^{z}; 1-\sigma_i, x, \D) d\sigma,
\end{equation*}
where $P_{\alpha_i}(L_0^{z}, L_{\alpha_i}^{z}; 1-\sigma_i, x, \D)$ is
defined in Lemma \ref{lemma.finite-BCH}.  Then
\begin{equation*}
  \Lambda_{\alpha, z} = \cP_{\alpha}(x,z,\D)e^{L_0^{z}}
\end{equation*}
where the product is the composition of operators and $\cP_{\alpha}$
is a differential operator of order $2k+\ell$ and polynomial degree
$\leq \ell=|\alpha| = \sum_{i=1}^k \alpha_i$.  More precisely, we can
write
\begin{equation} \label{eq.explicit}
  \cP_{\alpha}(x,z,\D) = \sum_{|\beta| \le \ell}\sum_{|\gamma| \le
    \ell + 2 k} a_{\beta, \gamma}(z)(x-z)^\beta \D^\gamma_x,
\end{equation}
with $a_{\beta,\gamma} \in \CIb(\RR^N)$ and $\beta$ and $\gamma$
multi-indices.
\end{lemma}

\begin{proof}
The proof is a calculation based on the repeated application of Lemma
\ref{lemma.finite-BCH} on $\Lambda_\alpha^{k,\ell}$.  We fix $\alpha
\in \fA_{k,\ell}$, and for simplicity we continue to denote
$P_m(\theta) = P_m(L_0^{z}, L_m^{z}; \theta, x, \D)$, when no
confusion can arise. Then,
\begin{align*}
\begin{split}
\Lambda_{\alpha, z}&= \int_{\Sigma_k} e^{(1-\sigma_1)L_0^{z}}
L_{\alpha_1} e^{(\sigma_1 - \sigma_2)L_0^{z} } L_{\alpha_2}
e^{(\sigma_{2} - \sigma_3)L_0^{z} } \cdots L_{\alpha_k} e^{\sigma_k
  L_0^{z}} d\sigma \\ & = \int_{\Sigma_k} P_{\alpha_1}(1-\sigma_1)
e^{(1-\sigma_2) L_0^{z}} L_{\alpha_2}e^{(\sigma_{2} - \sigma_3)L_0^{z}
} \cdots L_{\alpha_k} e^{\sigma_k L_0^{z}} d\sigma \\ & =
\int_{\Sigma_k} P_{\alpha_1}(1-\sigma_1) P_{\alpha_2}(1-\sigma_2)
e^{(1- \sigma_3)L_0^{z}} \cdots L_{\alpha_k} e^{\sigma_k L_0^{z}}
d\sigma \\ & = \int_{\Sigma_k} P_{\alpha_1}(1-\sigma_1)
P_{\alpha_2}(1-\sigma_2) \cdots P_{\alpha_k}(1-\sigma_k) e^{L_0^{z}} d
\sigma \\ & = \int_{\Sigma_k} \prod_{i=1}^k P_{\alpha_i}(1-\sigma_i)
e^{L_0^{z}} d\sigma \ = \left( \int_{\Sigma_k} \prod_{i=1}^k
P_{\alpha_i}(1-\sigma_i) d \sigma \right) e^{L_0^{z}}.
\end{split}
\end{align*}
The proof is complete.
\end{proof}

Finally, for $\ell\leq n$ we set
\begin{equation*}
\cP^{\ell}(x,z,\D) :=  \sum_{\alpha \in
  \fA_{\ell}} \cP_\alpha (x,z,\D) = \sum_{k=1}^\ell \sum_{\alpha \in
  \fA_{k, \ell}} \int_{\Sigma_{k}} \prod_{i=1}^k
P_{\alpha_i}(1-\sigma_i ) d\sigma,
\end{equation*}
so that
\begin{multline*}
  \Lambda^{\ell}_{z} = \sum_{k=1}^\ell \sum_{\alpha \in \fA_{k, \ell}}
  \Lambda_{\alpha, z} = \sum_{k=1}^\ell \sum_{\alpha \in \fA_{k,
      \ell}} \cP_{\alpha}(x,z,\D) e^{L_0^{z}} \\ = \sum_{k=1}^\ell
  \sum_{\alpha \in \fA_{k, \ell}} \int_{\Sigma_k} \prod_{i=1}^k
  P_{\alpha_i}(L_0^{z}, L_{\alpha_i}^{z}; 1-\sigma_i,x,\D) d\sigma
  e^{L_0^{z}} = \cP^{\ell}(x,z,\D) e^{L_0^{z}}.
\end{multline*}

A similar, but more complicated, representation holds also for
$\Lambda_{\alpha, z}$ and for multi-indices $\alpha \in \fA_{n+1,
  \ell}$. Indeed,
\begin{align*}
\begin{split}
& \Lambda_{\alpha, z} = \\ &= \int_{\Sigma_{n+1}}
  e^{(1-\sigma_1)L_0^{z}} L_{\alpha_1} e^{(\sigma_1 - \sigma_2)L_0^{z}
  } L_{\alpha_2} \cdots e^{(\sigma_{n} - \sigma_{n+1})L_0^{z} }
  L_{\alpha_{n+1}} e^{\sigma_{n+1} L^{s, z} } d\sigma \\ & =
  \int_{\Sigma_{n+1}} P_{\alpha_1}(1-\sigma_1) e^{(1-\sigma_2)
    L_0^{z}} L_{\alpha_2} \cdots e^{(\sigma_{n} - \sigma_{n+1})L_0^{z}
  } \cdots L_{\alpha_{n+1}} e^{\sigma_{n+1} L^{s, z} } d\sigma \\ & =
  \int_{\Sigma_{n+1}} P_{\alpha_1}(1-\sigma_1)
  P_{\alpha_2}(1-\sigma_2) e^{(1- \sigma_3)L_0^{z}} \cdots
  L_{\alpha_{n+1}} e^{\sigma_{n+1} L^{s, z} } d\sigma \\ & =
  \int_{\Sigma_{n+1}} P_{\alpha_1}(1-\sigma_1) \cdots
  P_{\alpha_k}(1-\sigma_{n+1}) e^{(1-\sigma_{n+1}) L_0^{z}}
  e^{\sigma_{n+1} L^{s, z} }d \sigma.
\end{split}
\end{align*}

We are now in the position to state the main result of this
section. Below, we set $\cP^{0} = 1$.  Let us recall the error term
\begin{equation}\label{eq.En}
  \bE_{n,n}^{s, z} := \sum_{\ell = n+1}^{(n+1)^2} s^{\ell-n-1}
  \Lambda^{\ell}_{z} = \sum_{\ell = n+1}^{(n+1)^2} \sum_ {k=1}^{n+1}
  \sum_{\alpha \in \fA_{k, \ell}} s^{\ell-n-1} \Lambda_{\alpha, z}
\end{equation}
introduced in Equation \eqref{eq.L.pert}.
{(There, we introduced $\bE_{d,n}$, but such error term is independent of $d$,
as long $d\geq n$, hence we can always assume that $d=n$.)

\begin{theorem}\label{thm.rep}
The perturbative expansion \eqref{eq.L.pert} of
$e^{L^{s,z}}$ can be written in the form
\begin{equation*}
  e^{L^{s,z}} = e^{L_0^{z}} + \sum_{\ell = 1}^{n} s^\ell
  \cP^{\ell}(x,z,\D)e^{L_0^z} + s^{n+1} \bE_{n, n}^{s,z},
\end{equation*}
where the differential operators $\cP^{\ell}$ are explicitly given by
Lemmas \ref{lemma.finite-BCH} and \ref{lemma.explicit}.
\end{theorem}

\begin{proof}  Starting with \eqref{eq.L.pert}, we have
\begin{multline*}
e^{L^{s,z}} = e^{L_0^{z}} + \sum_{\ell = 1}^n \sum_{k=1}^{\ell}
\sum_{\alpha \in \fA_{k, \ell}} s^\ell \Lambda_{\alpha, z} + s^{n+1}
\bE_{n,n}^{s,z}\\ = \sum_{\ell = 0}^n s^\ell \Lambda^{\ell}_{z} +
s^{n+1} \bE_{n,n}^{s,z} = e^{L_0^{z}} + \sum_{\ell = 1}^n s^\ell
\cP^{\ell}(x, z, \pa) e^{L_0^{z}} + s^{n+1} \bE_{n,n}^{s,z}.
\end{multline*}
This completes the proof.
\end{proof}

Recall that $e^{L_0^z}(x,y)$ is explicit given in equation \eqref{eq.Lokern}, since $z$
is arbitrary, but fixed, and it agrees with the function $G^0(z;x,y)$ defined by equation
\eqref{eq.def.G} in the  Introduction.

\begin{corollary} \label{cor.Gmuexplicit}
If $|\alpha|=\ell \leq n$, then the kernel of each operator  $\Lambda_{\alpha,z}$ appearing in the
perturbative expansion \eqref{eq.L.pert} is explicitly given by:
\[
         \mathfrak{P}^\ell(z,x, y) G(z;x,y),
\]
where the function $\mathfrak{P}^\ell$ are of the form
\begin{equation*}
    \mathfrak{P}^\ell(z,x,y) = \sum a_{\alpha, \beta}(z)(x-z)^\alpha(x-y)^{\beta},
\end{equation*}
with $|\alpha| \le \ell$,
$\beta \le 3\ell$, $a_{\alpha, \beta} \in \CIb(\RR^N)$.
\end{corollary}

 \begin{proof}
 We observe that $e^{tL^z_0}$ is a convolution operator, since $z$ is fixed, therefore
 \[
            \left(\cP_\alpha (x,z,\D)\, e^{L^z_0}\right)(x,y) = \cP_\alpha (x,z,\D)
            \left(e^{ L^z_0}(x,y)\right).
 \]
Then, the result follows from formula \eqref{eq.def.G} for $e^{ L^z_0}(x,y)$, formula
\eqref{eq.explicit} for $\cP_\alpha (x,z,\D)$, and the fact that $\cP^\ell$ is a sum of
such operators as $\alpha$ varies over
$\mathfrak{A}_{\ell}$.
(See also Lemma \ref{lemma.Lz} in the next section, Section \ref{sec.EE}.)
\end{proof}

The method introduced in this and the previous sections to approximate
the heat kernel of $L$ will be called the \NAME. Its description is now complete.

\section{Error estimates\label{sec.EE}}

In this final section, we prove all the bounds necessary to justify
the error estimate in the asymptotic expansion of Theorem
\ref{theorem.main1}. {\em Throughout this section, $n$ will denote the
order in the Taylor expansion of the coefficients of $L$,
which may differ from the approximation order as defined in equation \eqref{eq.L.pert}.
Such approximation order will be denoted by $\mu$, as in the statement of
Theorem \ref{theorem.main1}}. Recall
that the definition of the operators $\Lambda_{\alpha,z}$ depends, in
principle, on $n$. However, if $\alpha \in \fA_{k,\ell}$ and $n$ is
large ($n \ge k$, $n \ge \alpha_j$) the operator $\Lambda_{\alpha,z}$
no longer depends on $n$ (in which case it does not depend on $s$
either).  This observation, together with the fact that $n$ is fixed, justifies
omitting $n$ from the notation for $\Lambda_{\alpha,z}$.
Moreover,  the error terms $\bE_{d,\nu}^{s, z}$ are independent of $d$, as long as
$d\ge \nu$, which will always be the case, so we shall write $\bE_{\nu,\nu}^{s, z} = \bE_{d,\nu}^{s, z}$.
Below, we will use such error terms for $\nu=\mu$ and $\nu=n$, with $n>\mu$ appropriately chosen.

We start from Lemma \ref{lemma.L.pert}. All the terms appearing in
that lemma are operators with smooth distribution kernels by Corollary
\ref{cor.wel.def}.  We recall that we denote by $T(x,y)$ the
distribution kernel of an operator $T$ with smooth kernel (so $T(x,
y)$ is a smooth function such that $Tf(x) = \int_{\RR^N}
T(x,y)f(y)dy$). In terms of kernels, the formula of Theorem
\ref{thm.rep} takes the form
\begin{multline} \label{eq.one}
  e^{L^{s,z}}(x, y) = e^{L_0^{z}}(x, y) + \sum_{\ell = 1}^{\nu} s^\ell
  \Lambda_{z}^{\ell}(x, y) + s^{\nu+1} \bE_{\nu, \nu}^{s,z}(x, y) \\ =
  \sum_{\ell = 0}^{\nu} s^\ell \Lambda_{z}^{\ell}(x, y) + s^{\nu+1}
  \bE_{\nu, \nu}^{s,z}(x, y),
\end{multline}
where  again $\nu=\mu$ or $\nu=n$.

We recall that $L_{0}^{z}$ is obtained from $L$ by freezing the
coefficients of the highest order derivatives of $L$ at $z$ and by
discarding the lower order terms.

We now substitute  $x = z +
s^{-1}(x-z)$, $y = z + s^{-1}(y-z)$, and $z = z(x, y)$ in the Equation \eqref{eq.one} above, for some
function $z(x,y)$ to be specified later. Lemma
\ref{lemma.Green.funct.cor} and Equation \eqref{eq.one} then give
\begin{multline} \label{eq.two}
  	e^{s^2L}(x, y) = s^{-N} e^{L^{s,z}}(z + s^{-1}(x - z), z +
        s^{-1}(y-z)) \\ = \sum_{\ell = 0}^{\nu} s^\ell
        \Lambda_{z}^{\ell}(z + s^{-1}(x-z), z + s^{-1}(y-z)) \\ +
        s^{\nu+1} \bE_{\nu, \nu}^{s,z}(z + s^{-1}(x-z), z +
        s^{-1}(y-z)),
\end{multline}
which is valid for any $\nu\le n $, in particular for $\nu=\mu$ and for $\nu=n$.

Using the definition of the approximate Green function
$\cG_t^{[\mu,z]}(x,y)$, for $t = s^2$, in Equation \eqref{eq.main.def}, we
then obtain
\begin{equation} \label{eq.three}
  	e^{s^2L}(x, y) = \cG_{s^2}^{[\nu,z]}(x,y) + s^{\nu+1}
  	\bE_{\nu,\nu}^{s,z}(z + s^{-1}(x-z), z + s^{-1}(y-z)).
\end{equation}
The error term in the approximation defined by Equation
\eqref{eq.main.def} is consequently given by
\begin{equation}\label{eq.error.cE}
  	e^{s^2L}(x, y) - \cG_{s^2}^{[\nu,z]}(x,y) = s^{\nu+1}
  	\bE^{s,z}_{\nu,\nu}(z + s^{-1}(x - z), z + s^{-1}(y-z)) ,
\end{equation}
where $\bE^{s,z}_{\nu,\nu}$ is as in Equation \eqref{eq.L.pert} with $s =
\sqrt{t}$, and $z = z(x,y)$.

We next introduce the dilated error operator
\begin{equation}\label{eq.error3.1}
  \cE_{s^2}^{[\nu,z]} f(x) = \int_{\RR^N} \bE^{s,z}_{\nu,\nu}(z + s^{-1}(x - z), z
  + s^{-1}(y-z)) f(y) dy.
\end{equation}
 and define the approximation kernel  $\cG_{s^2}^{[\nu,z]}$ to be the operator with
kernel $\cG_{s^2}^{[\mu,z]}(x,y)$, so that
\begin{equation}\label{eq.error1}
  	e^{tL} - \cG_{t}^{[\nu,z]} = t^{(\nu+1)/2}\cE_t^{[\nu,z]}.
\end{equation}
We will use the above formula {\em only} for $\nu=\mu <n$, where $n$ will be taken large enough.

Indeed,
if $\mu<n$, then the error term can be written as,
\begin{multline}\label{eq.error2}
  \bE^{s,z}_{\mu,\mu}(z + s^{-1}(x - z), z + s^{-1}(y-z)) =
  \\ \sum_{\ell = \mu+1}^{(n+1)^2} \ \sum_ {k=1}^{ \max\{\ell,n+1\} }
  \ \sum_{\alpha \in \fA_{k, \ell}} s^{\ell-\mu-1} \Lambda_{\alpha,
    z}(z + s^{-1}(x - z), z + s^{-1}(y-z)).
\end{multline}
(See Equation \eqref{eq.full.sum}, for instance.)  We will estimate
$\bE_{\mu,\mu}^{s,z} = \bE_{n,\mu}^{s,z}$ by writing
\begin{equation} \label{eq.two.two}
  \bE_{\mu, \mu}^{s,z} = \sum_{\ell = \mu + 1}^{n} s^{\ell-\mu-1}
  \sum_{k = \mu+1}^{\ell} \sum_{\alpha \in \fA_{k, \ell}}
  \Lambda_{\alpha, z} + s^{n+1-\mu} \bE_{n, n}^{s,z}.
\end{equation}
The point of this formula is that the error term $\bE_{\mu,\mu}^{s, z}$ is independent of $n$,
as long as $\mu \le n$. However, splitting the error as done above will allow a better control
on the error estimate
of Theorem \ref{theorem.main1}. In fact, we will show that each $\Lambda_{z}^{\ell}$ in the
first sum, which does not depend on $s$, is a pseudodifferential operator, and its contribution
to the overall error after the parabolic rescaling will be obtained in terms of a refined analysis
on its symbol. This analysis, in turn, leads to some refined  estimates {\em uniformly in $s$} on
the norm of the operator between weighted Sobolev Spaces. On the other
hand, we will obtain only rough estimates on the remander term $\bE_{n,n}$, which will nevertheless
be enough, due to the additional factor $s^{n+1-\mu}$. The main issue in treating the remainder is
that some of its terms $ \Lambda_{\alpha, z}$ implicitly depend on $s$, a fact which  makes it
difficult to show the remainder is also a pseudodifferential operator, at least in the usual
H\"ormader class. It may be possible to show that  $\bE_{n,n}$ is indeed a pseudodifferential
operator employing more exotic symbol classes or amplitudes, but we do not need to pursue this
point here, since we are able to prove the {\em sharp} estimates of Theorem \ref{theorem.main1} in any case.

We now proceed along these lines.
The dilated error operator introduced in equation  \eqref{eq.error3.1} can be rewritten in
terms of approximation operators
\begin{equation}\label{eq.error4}
  \cL_{s,\alpha} f(x) =
   s^{-N}   \int_{\RR^N} \Lambda_{\alpha, z}(z + s^{-1}(x
  - z), z + s^{-1}(y-z))f(y) dy,
\end{equation}
as
\begin{equation}\label{eq.lemma.rough}
  \cE_{s^2}^{[\mu,z]} = \sum_{\ell = \mu+1}^{(n+1)^2} \ \sum_ {k=1}^{
    \max\{\ell,n+1\} } \ \sum_{\alpha \in \fA_{k, \ell}} s^{\ell + N  - \mu - 1}
  \cL_{s,\alpha}.
\end{equation}
We therefore obtain
\begin{equation} \label{eq.two.three}
  \cE_{t}^{[\mu,z]} = \sum_{\ell = \mu + 1}^{n} s^{\ell-\mu-1}
  \sum_{k = \mu+1}^{\ell} \sum_{\alpha \in \fA_{k, \ell}}
  \cL_{s,\alpha}+ s^{n+1-\mu} \cE_{t}^{[n,z]}.
\end{equation}

To evaluate $\|\cE_{s^2}^{[n,z]}f\|$ in a desired norm, it will then be
enough to evaluate each operator norm $\|\cL_{s,\alpha}\|$ (between
suitable Sobolev spaces). As explained above, we shall derive a rough
estimate for the terms with $\alpha \in \fA_{n+1, \ell}$ or some
$\alpha_i =n+1$ (which corresponds to $\Lambda_{\alpha,z}$ depending
on $s$). When $\Lambda_{\alpha,z}$ is independent of $s$ (that is for
$\alpha \in \fA_{k,\ell}$, $k \le n$, $\alpha_i \le n$), we shall
derive some more precise estimates. We begin with these refined, more
precise estimates.

\subsection{Precise estimates}
Recall that we denote by $\Lambda_{\alpha, z}(x, y)$ the distribution
kernel of the operator $\Lambda_{\alpha, z}$
since it is a smooth function. Thus, for $\alpha \in
\fA_{k,\ell}$, $k \le n$, $\alpha_i \le n$, $\Lambda_{\alpha, z}(x,
y)$ does not depend on $s$. Let us fix a function $z(x, y)$, which
will be specified later, and let $\cL_{s,\alpha}$ be the operator with
distribution kernel $$s^{-N}\Lambda_{\alpha, z}(z + s^{-1}(x - z), z +
s^{-1}(y-z)),$$ introduced above in Equation \eqref{eq.error4}, where
$z=z(x,y)$.

We will show below that in this range of $\alpha$}
for a suitable choice of the function $z$, the operator $\cL_{s,\alpha}$ is a
pseudodifferential operator whose symbol is well behaved. We shall
then use symbol calculus to derive the desired error estimates.  We
refer to \cite{Tay} for all relevant properties of pseudodifferential
operators. Below, we follow the usual convention and set
$D=\frac{1}{i} \pa$, ($i=\sqrt{-1}$), where if not specified otherwise $\pa=\pa_x$.

We shall need the standard seminorms
$p_{m,\alpha, \beta}$ given by
\begin{equation}
	p_{m,\alpha, \beta}(a) = \sup_{(x, \xi) \in \RR^N \times
          \RR^N} |\langle \xi \rangle^{|\beta| - m} \pa_x^\alpha
        	\pa_\xi^\beta a(x, \xi)|.
\end{equation}
Then the H\"ormander class $S^m_{1, 0}:= S^m_{(1,0)}(\RR^N \times
\RR^N)$, $m>-\infty$, is by definition the set of functions $a :
\RR^{2N} \to \CC$ satisfying $p_{m,\alpha,\beta}(a) < \infty$.  The
space $S^{-\infty}= S^{-\infty}(\RR^N \times \RR^N)$ is defined by the
same seminorms, but with $m \in \ZZ$ arbitrary.

We also denote by
\begin{equation} \label{eq.def.Fourier}
  	\cF u(x) = \hat u(\xi) := \int_{\RR^N} e^{-i\, \xi \cdot x} u(x) dx
\end{equation}
the usual Fourier transform of $u$.  For any symbol $a$ in the
H\"ormander class $S^{m}_{1, 0} := S^{m}_{1, 0}(\RR^N \times \RR^N)$,
we denote by  $a(x, D)$ the operator
\begin{equation}\label{eq.def.psdo}
  	a(x, D) u(x) = (2\pi)^{-N} \int_{\RR^N} e^{i\, x \cdot \xi} a(x,\xi)
	\hat u(\xi) d\xi,
\end{equation}
defined for $u$ in the Schwartz space $\cS(\RR^N)$.  We  will denote by
$\cF_2$ the Fourier transform in the second variable of a function of
two variables.  For $a \in S^{-\infty} := S^{-\infty}_{1, 0}(\RR^N
\times \RR^N)$, the operator $a(x, D)$ is smoothing with distribution
kernel
\begin{equation*}
  	a(x, D) (x,y) = (2\pi)^{-N} \int_{\RR^N} e^{i\, (x - y) \cdot
    	\xi} a(x, \xi) d\xi = (\cF_2^{-1}a)(x, x-y).
\end{equation*}
Let $K$ be a smooth function on $\RR^N \times \RR^N$. If  the
integral operator defined by $K$, which is smoothing,  is in fact a
pseudodifferential operator $a(x, D)$, then we can recover $a$ from
$K$ by the formula $(\cF_2^{-1}a)(x, y) = K(x, x-y)$, so
\begin{equation}
  a(x, \xi) = \int_{\RR^N} e^{-i\, \xi \cdot y} K(x, x-y) dy.
\end{equation}

Recall next the function $G(z; x) = (4\pi)^{-N/2} \det(A(z))^{-1/2}
e^{- x^T A(z)^{-1} x/4}$ introduced in Equation \eqref{eq.def.G}: Then
the distribution kernel of $e^{L^z_0}$ is given by
\begin{equation}
  	e^{L^z_0}(x, y) = G(z; x-y),
\end{equation}
and we have  the following result.

\begin{lemma} \label{lemma.Lz}
Let $z \in \RR^N$ be a parameter and let us consider the operator $T =
(x-z)^{\beta}\pa_x^{\gamma} e^{L_0^z}$, where $\beta$ and $\gamma$
are multi-indices. Then the distribution
kernel of $T$ is given by
\begin{equation*}
	T(x,y) = (x-z)^\beta (\pa_x^\gamma G)(z; x-y).
\end{equation*}
\end{lemma}

\begin{proof} The Lemma  follows from a direct computation.
\end{proof}

We will also need the following standard result.

\begin{lemma}\label{lemma.bounded}
\begin{enumerate}[(i)]
\item The Fourier transform in the second variable establishes an
isomorphism $\cF_2 : S^{-\infty} := S^{-\infty}(\RR^N \times \RR^N)
\to S^{-\infty}$.

\item Multiplication defines a continuous map $S^m_{(1, 0)} \times
S^{-\infty} \to S^{-\infty}$.

\item
If $\{a_{s}\}_{s\in (0,1]}$ is uniformly bounded in $S^{-\infty}$ and
$b_s(x, \xi) = a_s(x, s\xi)$, then the family $\{s^kb_{s}\}_{s\in (0,1]}$
is uniformly bounded in $S^{-k}_{1,0}$, $k \ge 0$.
\end{enumerate}
\end{lemma}

\begin{proof} This follows from a straightforward calculation.
\end{proof}

For our main result, we require some assumptions on the dilation
center $z$.

\begin{definition} A function $z : \RR^{2N} \to \RR^N$ will be called
{\em admissible} if
\begin{enumerate}[(i)]
\item $z(x,x) = x$, for all $x \in \RR^N$.
\item All derivatives of $z$ are bounded.
\end{enumerate}
\end{definition}

A typical example is $z(x,y) = \lambda x + (1-\lambda)y$, for some fixed
parameter $\lambda$.  A simple application of the mean value theorem
gives that $\<z-x\> \le C \<y-x\>$ for some constant $C>0$.

We are now ready to state and prove the main result of this subsection.

\begin{theorem}  \label{thm.mainbound}
Let $\alpha \in \fA_{k,\ell}$, $k \le n$, $\alpha \le n$.  Assume that $z : \RR^{2N} \to \RR^N$ is admissible. 
Then there exists a uniformly bounded
family $\{a_{s}\}_{s\in (0,1]}$ in $S^{-\infty}$ such that,
if  $b_s(x, \xi) := a_s(x, s\xi)$, then
\begin{equation*}
  	\cL_{s,\alpha} = b_{s}(x, D).
\end{equation*}
\end{theorem}

\begin{proof}
By Lemma \ref{lemma.explicit}, we have that $\Lambda_{\alpha, z}$ is a
finite sum of terms of the form $\varphi(z)(x-z)^\beta \pa_x^\gamma
e^{L_0^z}$ with $\varphi\in C^\infty_b$.
Let then $k_z(x, y)$ be the distribution kernel of
$a(z)(x-z)^\beta \pa_x^\gamma e^{L_0^z}$ and let
\begin{equation*}
  	K_s(x, y) := s^{-N}k_z(z + s^{-1}(x - z), z + s^{-1}(y-z)),\quad z =
  	z(x, y).
\end{equation*}
By abuse of notation, we shall denote also by $K_s$ the integral
operator defined by $K_s$. It is enough then to prove our theorem for
$K_s$. Namely, it is enough to show that there exists a uniformly
bounded family $\{a_{s}\}_{s\in (0,1]}$ in $S^{-\infty}$ such that
\begin{equation*}
  K_s = a_s(x, sD).
\end{equation*}

By lemma \ref{lemma.Lz}, we have that the distribution kernel of
$\pa_x^\gamma e^{L_0^z}$ is of the form $\psi(z, x-y)$
and belong to $S^{-\infty}$ as a function of $x-y$ for $z$ fixed.
(This is consistend with the fact that for each fixed
$z$, $\pa_x^\gamma e^{L_0^z}$ is a convolution operator.) More
precisely $\psi(z, x)$ is  $\cF_2(i\, \xi)^\gamma e^{-\xi^T \cdot A(z) \cdot \xi}$.
This observation implies
\begin{multline*}
  	K_s(x, y) = \varphi(z(x, y)) s^{-|\beta|-N} (x - z(x, y))^{\beta} \psi(z(x, y),
  	s^{-1}(x-y)) =: \\ \varphi(z) s^{-|\beta|-N} (x - z)^{\beta} \psi(z,
  	s^{-1}(x-y)), \quad z = z(x,y).
\end{multline*}

We then let
\begin{equation*}
  	b_s(x, \xi) = \int_{\RR^N} e^{-i\, y \cdot \xi} \phi(z)
   	s^{-|\beta|-N} (x - z)^{\beta} \psi(z,
	s^{-1} y) dy, \quad z = z(x,x-y).
\end{equation*}
Next, we observe that if we  change variables from $y$ to  $sy$, we can write \
$b_s(x, \xi) = a_s(x, s\xi)$, where
\begin{equation*}
  	a_s(x, \xi) = \int_{\RR^N} e^{-i\, y \cdot \xi} \phi(z)
   	s^{-|\beta|} (x - z)^{\beta} \psi(z, y) dy, \quad z = z(x,x-sy).
\end{equation*}
We need to show that $a_s$ is a bounded family in $S^{-\infty}$.
To this end, we observe that, since $\varphi \in C^\infty_b$ and the derivatives of $z$
are all bounded, $\varphi(z) \in S^1_{1, 0}$ as a function of $y$ for each $x$. Similarly,
for each $j=1,\dots,N$, $s^{-1}(x_j - z_j(x, x-sy)) \in S^1_{1,0}$ as a function of $y$
for fixed $x$,  and collectively they form bounded families for $s \in (0, 1]$. Lastly,
from what already observed above, $\psi(z,y) \in S^{-\infty}$ as a function of $y$ for
each fixed $x$. Therefore, $a_s\in S^{-\infty}$ uniformly in $s$  by Lemma \ref{lemma.bounded}.
The proof is complete.
\end{proof}

We now obtain the desired refined mapping property estimate by standard results on
pseudodifferential operators. Below, $t=s^2$.

\begin{theorem}\label{theorem.refined}
Let $\alpha \in \fA_{k,\ell}$, $k \le n$, $\alpha_j \le n$.
Assume that $z : \RR^{2N} \times \RR^N$ is admissible.
Then for any $1<p<\infty$, any $r\in \RR$,
\begin{equation}
   t^{k/2} \|\cL_{s, \alpha}\|_{W^{r,p} \to W^{r+k,p}} \le C_{k, r, p},
\end{equation}
for a constant $C_{k, r, p}$ independent of $t \in (0, 1]$.
\end{theorem}

\subsection{Rough estimates}
We now move to study the mapping properties of $\Lambda_{\alpha, z}$
when either $\alpha \in \fA_{n+1, \ell}$ or some $\alpha_i =n+1$. In this case,
the operators $\Lambda_{\alpha, z}$ depend on $s$ also, although this dependence is not
shown in the notation.

The mapping properties that we establish in this subsection
will allow us to obtain corresponding mapping properties for the error operator
$\bE_{n,n}^{s,z}$, which is not immediately in the form of a
pseudodifferential operator. Consequently, we are not able to derive
bounds as those in Theorem \ref{thm.mainbound} above. Nevertheless,
the bounds we derive are sufficient to
establish the {\em sharp} error estimates as $t\to 0^+$ in weighted Sobolev spaces for
the overall approximation, given in Theorem \ref{theorem.main1}. This result is achieved
by choosing judiciously  an $n$ large enough.

As before we denote $W^{r,p}_{a} = W^{r,p}_{a, w}$ as before, where $w$
is the center of the weight $\<x\>_w = \<x-w\>=\<w-x\>$ used to define the
exponentially weighted Sobolev spaces  (see equation \eqref{def.w.S}). We shall also write
$L^p_a = W^{0,p}_a$. The main result of this section is the following proposition.

\begin{proposition}\label{prop.rough} Assume that $z : \RR^{2N} \to \RR^N$ is admissible.
For any $\alpha$, any $1<p<\infty$, $k\in \ZZ_+$, $r\ge 0$, and $a \in \RR$,
\begin{equation}
   	s^{k} \|\cL_{s, \alpha}\|_{L^{p}_{a} \to W^{k,p}_{a}} \le C_{k, p},
\end{equation}
for a constant $C_{k, p}$ independent of $s \in (0, 1]$, of $a$ in
a bounded set, and independent of the center of the weight that defines
the weighted Sobolev spaces $W^{k,p}_a$.
\end{proposition}

\begin{proof}
The proof is based on explicit kernel estimates and Riesz' lemma.
By replacing the operator $L$
with $e^{a\<x-w\>} L e^{-a\<x-w\>}$, where $w$ is the center of the
weight, we can assume that $a=0$, as before.

As before,  $\Lambda_{\alpha, z}(x, y)$ is the smooth distribution kernel
of the operator $\Lambda_{\alpha, z}$. For any given point $v\in \RR^N$, we denote by
$\delta_{v}^{\beta}$ the distribution defined by
$\delta_{v}^{\beta}(f) = \D^\beta f(v)$ (we agree that $\delta_v^0(f)
= f(v)$.  Then
\begin{equation}\label{ker.est}
    	\partial_x^\beta \partial_y^{\beta'} \partial_z^{\beta''}
	\Lambda_{\alpha, z}(x, y) = \< \delta_x^\beta\ , \
	(\partial_z^{\beta''} \Lambda_{\alpha,z}) (\delta^{\beta'}_y)
	\>,
\end{equation}
where $\<,\>$ is the usual duality pairing.
Since all the coefficients (and their derivatives) of $L$ are uniformly bounded,
the derivative $\partial_z^\beta \Lambda_{\alpha,z}$ will satisfy
the same mapping properties as $
\Lambda_{\alpha,z}$.  Furthermore, for each
multi-index $\beta$, $\partial^\beta \delta_y \in H^{-q}(\RR^N)$ for $q
> N/2 + |\beta|$ and has norm independent of $y$.

In the rest of the proof, we  use the weighted Sobolev spaces
introduced in \eqref{def.w.S}. We recall that the mapping properties
between these spaces are uniform in term of the base point. We can
therefore choose the weight center at $x$ in estimating \eqref{ker.est}.
We will write $H^s_a = W^{s,2}_{a,x}$.  Then $\delta_y\in H^{-q}_a$ for all $a\in
\RR$, $q > N/2 + |\beta|$,  with
\begin{equation*}
  \|\partial^\beta \delta_y\|_{H^{-q}_{a}} :=
  \|e^{a<y-x>}\partial^\beta \delta_y\|_{H^{-q}} \le C_{q,\alpha}
  e^{(a+ \epsilon)  \<y-x\>}.
\end{equation*}

Next, we pick an  $\epsilon>0$ small enough.
Replacing $\epsilon$ with
$\epsilon/k$, where $k$ is such that $\alpha\in \mathfrak{U}_{k,\ell}$ in Corollary \ref{cor.wel.def} yields
$$
\|\partial_z^{\beta}
\Lambda_{\alpha,z}
\, \|_{H^{-q}_{-a} \to H^{q}_{-a- \epsilon/k}} \| \le e^{\epsilon \<x-z\>},
$$
and hence
\begin{multline}\label{eq.mult.abc}
	\big|\partial_x^{\beta} \partial_z^{\beta'}
	\partial_y^{\beta''} \Lambda_{\alpha,z}(x, y)\big| = \big| \<
	\partial^{\beta} \delta_x\, ,
	\partial_z^{\beta'}\Lambda_{\alpha,z}\, \partial^{\beta''}
	\delta_y \> \big| \\ \le C \|\partial^{\beta}
	\delta_x\|_{H^{-q}_{-a-\epsilon}/k}
	\|\partial_z^{\beta'} \Lambda_{\alpha,z}
	\,\|_{H^{-q}_{-a}\to H^{q}_{-a-\epsilon/k}} \|\partial^{\beta''}
	\delta_y\|_{H^{-q}_{-a}}\\  \le C e^{\epsilon \<x-z\> -
	(a+\epsilon)\<y-x\>},
\end{multline}
where $q>N/2+\max(|\beta|,|\beta'|,|\beta''|)$.

We will employ the bounds above to estimate
\begin{equation}\label{eq.error4.1}
  	\cL_{s,\alpha} (x,y) = s^{-N} \Lambda_{\alpha, z}
	(z+s^{-1}(x-z), z + s^{-1}(y-z)), \quad z = z(x,y).
\end{equation}
We first use the chain rule to conclude that, if $\gamma$ is any
multi-index, then $\pa_x^\gamma
\cL_{s,\alpha} (x,y)$ is a sum of
terms of the form
\begin{equation*}
	s^{-j} \partial_x^{\beta} \partial_z^{\beta'}
	\partial_y^{\beta''}
	\Lambda_{\alpha,z}(z+s^{-1}(x-z), z + s^{-1}(y-z)) P,
\end{equation*}
for appropriate multi-indices $\beta$, $\beta'$, and $\beta''$,
with $P$ a product of factors of the form $\partial^{\alpha'}z$ and $j
\le |\gamma|$.  Our assumptions on $z$ imply that $p$ is
bounded. Using also Equation \eqref{eq.mult.abc}, we obtain for $\epsilon$ sufficiently small,
\begin{multline}\label{eq.error4.2}
  	\big| \pa_x^\gamma \cL_{s,\alpha} (x,y) \big| \le C
	s^{-N-|\gamma| } \ e^{\epsilon \<s^{-1}(x-z)\> -
	a\<s^{-1}(y-x)\>} \\ \le C s^{-N-|\gamma| } \ e^{- a
	\<s^{-1}(y-x)\>/2},  \qquad z = z(x,y),
\end{multline}
where the last inequality follows from $\<x-z\> \le C \<y-x\>$.  {From}
this inequality, we obtain after the change of variables $v = s^{-1}(y-x)$
\begin{equation*}
  \int_{\RR^N} \big|\pa_x^\gamma \cL_{s,\alpha} (x,y) \big| dy \le
  C_a s^{-|\gamma|}, \quad \forall x \in \RR^N,
\end{equation*}
and
\begin{equation*}
  	\int_{\RR^N} \big|\pa_x^\gamma \cL_{s,\alpha} (x,y) \big| dx \le
  	C_a s^{-|\gamma|}, \quad \forall y \in \RR^N,
\end{equation*}
These two estimates together with Riesz Lemma give that the map $f \to
s^{|\gamma|} \pa_x^\gamma \cL_{s,\alpha} f$ is bounded from $L^p$ to
$L^p$, which is enough to establish the result.
\end{proof}

This proposition,
and the definition of
$\bE_{n,n}^{s,z}$ immediately imply the following lemma, where as usual $t=s^2$.

\begin{lemma}\label{lemma.rough}
Assume that $z : \RR^{2N} \times \RR^N$ is admissible, then
for each $r\in \RR$, $q>0$, we have
\begin{equation*}
  	\| \cE_t^{[n,z]}\|_{W^{r,p} \to W^{r+q	,p}} \le C_T \, t^{-(r+q)/2}, \qquad t
  \in (0,T].
\end{equation*}
\end{lemma}

\begin{proof}
Indeed, this follows from Proposition \ref{prop.rough}, Equation
\eqref{eq.lemma.rough}, and the continuous inclusion $W^{r,p}\hookrightarrow
L^p$, $r\geq 0$. For $r$ noninteger we also use interpolation.
\end{proof}

We note that in the above proposition we have an additional factor of $t^{-q/2}$
compared with the refined estimates of Theorem \ref{theorem.refined}.
This extra factor will not affect the final result, however, provided the order $n$
of the Taylor expansion of $L$ is chosen sufficiently large.

Then, the lemma leads to the following more precise estimate for the error operator $\cE_{t}^{[\mu,z]}$.

\begin{theorem}\label{theorem.rough} Assume that $z : \RR^{2N} \to \RR^N$ is admissible, then we have
\begin{equation*}
  	\| \cE_{t}^{[\mu,z]}\|_{W^{r,p} \to W^{r+k,p}} \le C_T\, t^{- k/2},
	\qquad t \in (0,T].
\end{equation*}
\end{theorem}

\begin{proof} Let us chose $n +1 \ge \mu + r$ and $t = s^2$, as usual. Then
Theorems \ref{theorem.refined} and \ref{theorem.rough} applied to
Equation \eqref{eq.two.three} give
\begin{multline*}
  	\|\cE_{t}^{[\mu,z]}\|_{W^{r,p} \to W^{r+k,p}} \le
	\sum_{\ell = \mu + 1}^{n} s^{\ell-\mu-1}
  	\sum_{k = \mu+1}^{\ell} \sum_{\alpha \in \fA_{k, \ell}}
  	\|\cL_{\alpha, z}\|_{W^{r,p} \to W^{r+k,p}}\\  + s^{n+1-\mu}
	\|\cE_{t}^{[n,z]}\|_{W^{r,p} \to W^{r+k,p}}
  	\le Cs^{-k} (1  + s^{n+1-\mu}s^{-r-k}) \le Cs^{-k} .
\end{multline*}
\end{proof}

This completes the proof of Theorem \ref{theorem.main1}.

{From} \eqref{eq.main.def}, we immediately obtain the following property
on the principal part of the asymptotic expansion.

\begin{corollary} \label{cor.maintermbound}
Assume that $z : \RR^{2N} \to \RR^N$ is admissible.  For each
$1<p<\infty$, $r\in \RR$, $\mu \ge 0$, and any $f\in W^{r,p}_{a}$ let
us define
\begin{equation*}
    	\cG^{[\mu,z]}_t f(x) := \int_{\RR^N} \cG^{[\mu,z]}_t(x,y)
        f(y)\, dy,
 \end{equation*}
then $\cG^{[\mu,z]}_t f \to f$ in $W^{r,p}_{a}$ for $t \to 0_+$.
\end{corollary}


\begin{thebibliography}{CMN}
%
\bibitem{Ait}{\sc Y. Ait-Sahalia}, {\em Closed-form likelihood expansions for multivariate diffusions},
The Annals of Statistics, (2008), Vol. 36, No. 2, 906--937.
%
\bibitem{ALN} {\sc B. Ammann, R. Lauter, \& V. Nistor}, {\em On the
 geometry of Riemannian manifolds with a Lie structure at infinity}.
 Int. J. Math. Math. Sci. {\bf 2004}, no. 1-4, 161--193.
%
\bibitem{AL}{\sc M. Avellaneda \& P. Laurence} {\em Quantitative Modeling of Derivative Securities:
{From} Theory To Practice}, CRC Press, 1999.
%
\bibitem{Az}{\sc R. Azencott}, {\em Asymptotic small time expansions for densities of diffusion processes},
Lecture Notes Maths, {\bf1059},     pp 402-498, Springer-Verlag, 1984.
%
\bibitem{B}{\sc H. Baker}, Proc Lond Math Soc (1) 34 (1902) 347-360;
ibid (1) 35 (1903) 333-374; ibid (Ser 2) 3 (1905) 24-47.
%
\bibitem{BPV}{\sc E. Barucci, S. Polidoro, \& V. Vespri},  {\em Some results on partial differential equations and Asian options},
 Math. Models Methods Appl. Sci.  11  (2001),  no. 3, 475--497.
%
\bibitem{BenA}{\sc G. Ben Arous}, {\em Flots et series de Taylor stochastiques}, Probab. Theory Relat.
Fields, {\bf 81}, 29--77 (1989).
%
\bibitem{BL}{\sc J. Bergh, J. L\"ofstr\"om}, {\em Interpolation
spaces. An introduction}. Grundlehren der Mathematischen
Wissenschaften, No. 223. Springer-Verlag, Berlin-New York, 1976.
%
\bibitem{BS}{\sc F. Black \& M. Scholes}, {\em The pricing of options and corporate
liabilities}, The Journal of Political Economy, Volume 81, Issue 3, (May - June 1973), 637-654.
%
\bibitem{CFP}{\sc F.Corielli, P. Foschi, A. Pascucci}, {\em Parametrix approximation of diffusion transition
densities}, Preprint, 2009.
%
\bibitem{Hundertmark} {\sc K. Broderix, D. Hundertmark, H. Leschke},
\emph{Continuity properties of Schr\"odinger semigroups with magnetic fields. (English summary)}
Rev. Math. Phys. \textbf{12} (2000), no. 2, 181--225. 
%
\bibitem{C}{\sc J. Campbell}, Proc Lond Math Soc 28 (1897) 381--390; ibid 29 (1898) 14--32.
%
\bibitem{Carmona} {\sc R. Carmona},
\emph{Regularity properties of Schr\"{o}dinger and Dirichlet semigroups},
J. Funct. Anal. \textbf{33} (1979), no. 3, 259--296. 
%
\bibitem{CN}{\sc R. Carmona \& S. Nadtochiy}, {\em An infinite dimensional stochastic analysis approach
to local volatility dynamic models}, Communications on Stochastic Analysis, 2(1), 2008.
%
\bibitem{Carmichael} {\sc H. J. Carmichael},  {\em Statistical methods in quantum optics. 1.
Master equations and Fokker-Planck equations}. Texts and Monographs in Physics. Springer-Verlag, Berlin, 1999.
%
\bibitem{Cast}{\sc F. Castell}, {\em Asymptotic expansion of stochastic flows}, Probability Theory and
Related Fields, {\bf 96}, No. 2, (1993), pp 225-239.
%
\bibitem{CGT}{\sc J. Cheeger, M. Gromov, M. E.  Taylor}, {\em Finite
propagation speed, kernel estimates for functions of the Laplace
operator, and the geometry of complete Riemannian
manifolds}. J. Differential Geom. {\bf 17} (1982), no. 1, 15--53.
%
\bibitem{CCCMN}{\sc W. Cheng, R. Constantinescu, N. Costanzino,
A. Mazzucato, V. Nistor}, {\em Approximate Solutions to Second Order
Parabolic Equations III: manifolds with bounded geometry}, in preparation.
%
\bibitem{CCLMN}{\sc W. Cheng, N. Costanzino,
J. Liechty, A. Mazzucato, V. Nistor}, {\em Closed form asymptotics for local volatility models}, Preprint.
%
\bibitem{CCCMN2}{\sc W. Cheng, N. Costanzino,
R. Constantinescu, A. Mazzucato, V. Nistor}, {\em Closed form asymptotics for stochastic volatility models}, Preprint.
%
\bibitem{CMN} {\sc W. Cheng, A. Mazzucato, V. Nistor}, {\em Approximate Solutions to Second Order
Parabolic Equations II: time-dependent operators}, in final prepration.
%
\bibitem{DiBenedetto}{\sc E. DiBenedetto}, {\em Partial Differential
Equations}, Birkh\"auser, Boston, MA, 1995.
%
\bibitem{Evans}{\sc L.C. Evans}, {\em Partial Differential Equations},
Grad. Stud. Math., vol. 19, Amer. Math. Soc., Providence, RI, 1998.
%
\bibitem{Farkas} {\sc W. Farkas,  N. Reich, C. Schwab},  \emph{Anisotropic stable L\'{e}vy copula 
processes---analytical and numerical  aspects}, 
 Math. Models Methods Appl. Sci.  \textbf{17}  (2007),  no. 9, 1405--1443.
 %
\bibitem{FPS}{\sc J.P. Fouque, G. Papanicolaou, K.R. Sircar}, {\em Derivatives in Financial Markets with
Stochastic Volatility}, Cambridge University Press, 2000.
%
\bibitem{Gardiner} {\sc C.W. Gardiner}, {\em Handbook of Stochastic Methods: for Physics, Chemistry and
the Natural Sciences}. Third Ed. Series in Synergetics 13. Springer-Verlag,  Berlin,  2004.
%
\bibitem{Gatheral}{\sc J. Gatheral}, {\em The Volatility Surface: A Practitioner's Guide}, John Wiley
\& Sons, 2006.
%
\bibitem{JaffeBook} {\sc J. Glimm and A. Jaffe},
{\em Quantum physics. A functional integral point of view.} Second edition. Springer-Verlag,
 New York, 1987.
%
\bibitem{Greiner}{\sc P. Greiner}, {\em An asymptotic expansion for
the heat equation}. Arch. Rational Mech. Anal. {\bf 41} (1971),
163--218.
%
\bibitem{H}{\sc F. Hausdorff}, Ber Verh Saechs Akad Wiss Leipzig, {\bf 58}
(1906) 19-48.
%
\bibitem{Heston}{\sc S.L. Heston},  {\em A closed-form solution for options with
stochastic volatility with applications to bond and currency options}, The
Review of Financial Studies, Vol 6, No. 2, (1993), 327-343.
%
\bibitem{JoachainBook} {\sc C. J. Joachain}, Quantum collision theory,
North-Holland Publishing (Elsevier), 1975.
\bibitem{Kampen} {\sc J. Kampen},  {\em On the WKB-expansion of parabolic equations and Applications},
SSRN (2006).
%
\bibitem{KatoBook} {\sc T. Kato},
Perturbation theory for linear operators. Reprint of the 1980 edition.
Classics in Mathematics. Springer-Verlag, Berlin, 1995.
%
\bibitem{Koch1}{\sc H. Koch}, {\em Partial differential equations with
non-Euclidean geometries}.  Discrete Contin. Dyn. Syst. Ser. S {\bf 1}
(2008), no. 3, 481--504.
%
\bibitem{KT}{\sc H. Koch, D. Tataru}, {\em Well-posedness for
the Navier-Stokes equations}.  Adv. Math. {\bf 157} (2001), no. 1,
22--35.
%
\bibitem{Krainer}
{\sc T. Krainer}, \emph{Maximal {$L\sp p$}-{$L\sp q$} regularity for parabolic partial
differential equations on manifolds with cylindrical ends}, 
Integral Equations Operator Theory, \textbf{63} (2009), 521--531.
%
\bibitem{Lunardi}{\sc A. Lunardi}, {\em Analytic semigroups and
optimal regularity in parabolic problems}. Progress in Nonlinear
Differential Equations and their Applications, 16. Birkh\" auser
Verlag, Basel, 1995.
%
\bibitem{Lesniewski}{\sc P. Hagan, D. Kumar, A. S. Lesniewski, \& D. E. Woodward},
{\em Managing smile risk}, Willmott Magazine, (2002) September, 84-108.
%
\bibitem{Lew}{\sc A.L. Lewis}, {\em Option valuation under stochastic volatility with Mathematica code},
Newport Beach, California: Finance Press, (2000).
%
\bibitem{MN} {\sc A.L. Mazzucato \& V. Nistor}, {\em Mapping
properties of heat kernels, maximal regularity, and semi-linear
parabolic equations on noncompact manifolds}. Journal of Hyperbolic
Differential Equations {\bf 3} (2006), n. 4, 599-629.
%
\bibitem{McKeanSinger}{\sc H. P. McKean, \& I. M. Singer}, {\em Curvature
and the eigenvalues of the Laplacian}. J. Differential Geometry {\bf
1} (1967), no. 1, 43--69.
%
\bibitem{Melrose2}{\sc R. Melrose},{\em The Atiyah-Patodi-Singer index
theorem}. Research Notes in Mathematics {\bf 4}. A K Peters, Ltd.,
Wellesley, MA, 1993.
%
\bibitem{BochnerBook}{\sc J. Mikusi\' nski}, {\em The Bochner integral},
LehrbŸcher und Monographien aus dem Gebiete der exakten Wissenschaften, Mathematische Reihe,
Band 55. Birkh\" auser
Verlag, Basel-Stuttgart, 1978.
 %
\bibitem{Pazy}{\sc A. Pazy}, {\em Semigroups of linear operators and
applications to partial differential equations}. Applied Mathematical
Sciences, 44. Springer-Verlag, New York, 1983.
%
\bibitem{Pleijel} {\sc S. Minakshisundaram \& A. Pleijel},
{\em  Some properties of the eigenfunctions of the Laplace-operator on
 Riemannian manifolds.}  Canadian J. Math.  {\bf 1},  (1949). 242--256.
%
\bibitem{Shubin}{\sc M.A. Shubin}, {\em Spectral theory of elliptic
operators on noncompact manifolds.  Methodes semi-classiques}, Vol. 1
(Nantes, 1991).  Asterisque {\bf 207} (1992), no. 5, 35--108.
%
\bibitem{Simon} {\sc B. Simon},
\emph{Schr\"{o}dinger semigroups}, 
Bull. Amer. Math. Soc. (N.S.) \textbf{7} (1982), no. 3, 447--526. 
\bibitem{Hsu}{\sc E. P. Hsu}, {\em Stochastic Analysis on Manifolds},
Graduate Studies in Mathematics, Vol 38, (2002).
%
\bibitem{TayPDEII}{\sc M.E. Taylor}, {\em Partial differential
equations. II. Qualitative studies of linear equations}. Applied
Mathematical Sciences, 116. Springer-Verlag, New York, 1996.
\bibitem{Tay}{\sc M.E. Taylor}, {\em Pseudodifferential operators},
Princeton Mathematical Series, 34. Princeton University Press,
Princeton, N.J., 1981.
%
\bibitem{T}{\sc M.E. Taylor}, {\em Pseudodifferential operators and
Nonlinear PDE}, Birkh\"{a}user, Boston 1991.
%
\bibitem{Trecent} {\sc M.E. Taylor}, {\em Hardy spaces and BMO on
manifolds with bounded geometry}. J. Geom. Anal. {\bf 19} (2009),
no. 1, 137--190.
%
\bibitem{Triebel}{\sc H. Triebel}, {\em Theory of function
spaces. II}. Monographs in Mathematics, 84. Birkhauser Verlag,
Basel, 1992.
%
\bibitem{V1}{\sc S.R.S. Varadhan}, {\em Diffusion processes in a small time interval},
Comm.  Pure Appl. Math. 20 (1967), 659--685.
%
\bibitem{V2}{\sc S.R.S. Varadhan}, {\em On the behavior of the fundamental solution of the
heat equation with variable coefficients},
Comm. Pure Appl. Math. 20 (1967), 431--455.
%
\bibitem{Vas}{\sc D.V. Vassilevich}, {\em Heat kernel expansion: User's manual}, Physics
Reports, 388:279--360, (2003).
%
\bibitem{Wilcox}{\sc R.M. Wilcox}, {\em Exponential operators and
parameter differentiation in Quantum Physics}, J. Math. Phys., {\bf 8}
)1967), 962-982.
%
\bibitem{Yosida}{\sc K. Yosida}, {\em Functional analysis}. Reprint of
the sixth (1980) edition. Classics in Mathematics. Springer-Verlag,
Berlin, 1995.
\end{thebibliography}
\end{document}